\documentclass[final,leqno]{siamltex}

\usepackage[numbers,square]{natbib}
\usepackage{amsmath}
\usepackage{amsfonts}
\usepackage{amssymb}
\usepackage{commath} 
\usepackage{graphicx}
\usepackage[lofdepth,lotdepth]{subfig}
\usepackage{hyperref}

\title{Analysis of block-preconditioners for models of coupled
  magma/mantle dynamics}

\author{Sander Rhebergen\thanks{Mathematical Institute,
    University of Oxford, Andrew Wiles Building, Radcliffe Observatory
    Quarter, Woodstock Road, Oxford OX2 6GG, United Kingdom and
    Department of Earth Sciences, University of Oxford, South Parks
    Road, Oxford OX1 3AN, United Kingdom ({\tt
      sander.rhebergen@maths.ox.ac.uk}).}  \and Garth
  N. Wells\thanks{Department of Engineering, University of Cambridge,
    Trumpington Street, Cambridge CB2 1PZ, United Kingdom ({\tt
      gnw20@cam.ac.uk}).} \and Richard F. Katz\thanks{Department of
    Earth Sciences, University of Oxford, South Parks Road, Oxford OX1
    3AN, United Kingdom ({\tt richard.katz@earth.ox.ac.uk}).}  \and
  Andrew J. Wathen\thanks{Mathematical Institute, University of
    Oxford, Andrew Wiles Building, Radcliffe Observatory Quarter,
    Woodstock Road, Oxford OX2 6GG, United Kingdom ({\tt
      andy.wathen@maths.ox.ac.uk}).} }

\begin{document}
\maketitle
\begin{abstract}
  This article considers the iterative solution of a finite element
  discretisation of the magma dynamics equations. In simplified form,
  the magma dynamics equations share some features of the Stokes
  equations. We therefore formulate, analyse and numerically test a
  Elman, Silvester and Wathen-type block preconditioner for magma
  dynamics. We prove analytically and demonstrate numerically the
  optimality of the preconditioner. The presented analysis highlights
  the dependence of the preconditioner on parameters in the magma
  dynamics equations that can affect convergence of iterative linear
  solvers. The analysis is verified through a range of two- and
  three-dimensional numerical examples on unstructured grids, from
  simple illustrative problems through to large problems on subduction
  zone-like geometries. The computer code to reproduce all numerical
  examples is freely available as supporting material.
\end{abstract}
\begin{keywords}
  Magma dynamics, mantle dynamics, finite element method,
  preconditioners.
\end{keywords}
\begin{AMS}
  65F08, 76M10, 86A17, 86-08.
\end{AMS}
\pagestyle{myheadings}
\thispagestyle{plain}
\markboth{S. RHEBERGEN, G.~N.~WELLS, R.~F. KATZ AND A.J.~WATHEN}%
{ANALYSIS OF PRECONDITIONERS FOR COUPLED MAGMA/MANTLE DYNAMICS}

\section{Introduction}

The mantle of Earth extends from the bottom of the crust to the top of
the iron core, some 3000~km below. Mantle rock, composed of silicate
minerals, behaves as an elastic solid on the time scale of seismic
waves but over geological time the mantle convects at high Rayleigh
number as a creeping, viscous fluid~\citep{Schubert:book}. This
convective flow is the hidden engine for plate tectonics, giving rise
to plate boundaries such as mid-ocean ridges (divergent) and
subduction zones (convergent). Plate boundaries host the vast majority
of terrestrial volcanism; their volcanoes are fed by magma extracted
from below, where partial melting of mantle rock occurs (typically at
depths less than $\sim$100~km).

Partially molten regions of the mantle are of interest to
geoscientists for their role in tectonic volcanism and in the chemical
evolution of the Earth.  The depth of these regions makes them
inaccessible for direct observation, and hence studies of their
dynamics have typically involved numerical simulation. Simulations are
often based on a system of partial differential equations derived
by~\citet{McKenzie:1984} and since elaborated and generalised by other
authors~\citep[e.g.,][]{Bercovici:2003, Simpson:2010a,
  Simpson:2010b}. The equations describe two interpenetrating fluids
of different density and vastly different viscosity: solid and molten
rock (i.e.,~mantle and magma). The grains of the rock form a viscously
deformable, permeable matrix through which magma can percolate. This
is captured in the theory by a coupling of the Stokes equations for
the mantle with Darcy's law for the magma. Although each phase is
independently incompressible, the two-phase mixture allows for
divergence or convergence of the solid matrix, locally increasing or
decreasing the volume fraction of magma. This process is modulated by
a compaction viscosity, and gives rise to much of the interesting
behaviour associated with coupled magma/mantle
dynamics~\cite{Spiegelman:1993a, Spiegelman:1993b, Katz:2006,
  Takei:2013}.

The governing equations have been solved in a variety of contexts,
from idealised studies of localisation and wave behaviour
\cite[e.g.][]{Aharonov:1995, Barcilon:1986} to applied studies of
plate-tectonic boundaries, especially mid-ocean ridges
\cite[e.g.][]{Ghods:2000, Katz:2008}. These studies have employed
finite volume techniques on regular, Cartesian grids
\cite[e.g.][]{Katz:2007}. Unlike mid-ocean ridges, subduction zones
have a plate geometry that is awkward for Cartesian grids; it is,
however, conveniently meshed with triangles or tetrahedra, which can
also focus resolution where it is most needed
\cite{Keken:2008}. Finite element simulations of pure mantle
convection in subduction zones are common in the literature, but it
remains a challenge to model two-phase, three-dimensional,
magma/mantle dynamics of subduction, even though this is an area of
active research~\citep{Keller:2013, Wilson:2013b}. Such models require
highly refined computational meshes, resulting in very large systems
of algebraic equations. To solve these systems efficiently, iterative
solvers together with effective preconditioning techniques are
necessary. Although the governing equations are similar to those of
Stokes flow, there has been no prior analysis of their discretisation
and numerical solution by the finite element method.

The most computationally expensive step in modelling the partially
molten mantle is typically the solution of a Stokes-like problem for
the velocity of the solid matrix.  To address this bottleneck in the
context of large, unstructured grids for finite element
discretisations, we describe, analyse, and test a preconditioner for
the algebraic system resulting from the simplified McKenzie equations.
The system of equations is similar to the Stokes problem, for which
the Silvester--Wathen preconditioner \cite{Silvester:1994} has been
proven to be optimal, i.e., the iteration count of the iterative
method is independent of the size of the algebraic system for a
variety of discretisations of the Stokes equations (see also
\cite{May:2008}). The key lies in finding a suitable approximation to
the Schur complement of the block matrix resulting from the finite
element discretisation. We follow this approach to prove and
demonstrate numerically the optimality of the preconditioner for
coupled magma/mantle dynamics problems. The analysis and numerical
examples highlight some issues specific to magma/mantle dynamics
simulations regarding the impact of model parameters on the solver
performance. To the best of our knowledge, together with the work of
\citet{Katz:2013}, we present the first three dimensional computations
of the (simplified) McKenzie equations, and the first analysis of a
preconditioner for this problem.

In this work we incorporate analysis, subduction zone-inspired
examples, and software implementation. The analysis is confirmed by
numerical examples that range from illustrative cases to large,
representative models of subduction zones solved using parallel
computers. The computer code to reproduce all presented examples is
parallelised and is freely available under the Lesser GNU Public
License (LGPL) as part of the supporting
material~\citep{supporting-material}.  The proposed preconditioning
strategies have been implemented using libraries from the FEniCS
Project~\citep{alnaes:2013,logg:2010,fenics:book,oelgaard:2010} and
PETSc~\citep{petsc-efficient,petsc-user-ref,petsc-web-page}. The
FEniCS framework provides a high degree of mathematical abstraction,
which permits the proposed methods to be implemented quickly,
compactly and efficiently, with a close correspondence between the
mathematical presentation in this paper and the computer
implementation in the supporting material.

The outline of this article is as follows. In
Section~\ref{s:magmaDynamics} we introduce the simplified McKenzie
equations for coupled magma/mantle dynamics, followed by a finite
element method for these equations in Section~\ref{s:fem}. A
preconditioner analysis is conducted in Section~\ref{s:precond_stokes}
and its construction is discussed in
Section~\ref{s:pcconstruction}. Through numerical simulations in
Section~\ref{s:numsim} we verify the analysis; conclusions are drawn
in Section~\ref{s:conclusions}.

\section{Partially molten magma dynamics}
\label{s:magmaDynamics}

Let $\Omega \subset \mathbb{R}^d$ be a bounded domain with $2 \le d
\le 3$. The \citet{McKenzie:1984} model on $\Omega$ reads
\begin{align}
  \label{eq:pf_mckenzie_a}
  \partial_{t} \phi - \nabla \cdot \del{(1 - \phi)\mathbf{u}} &= 0,
  \\
  \label{eq:pf_mckenzie_b}
  -\nabla \cdot 2\eta \boldsymbol{\epsilon}(\mathbf{u}) + \nabla p_{\rm f}
  &= \nabla \del{\del{\zeta - \tfrac{2}{3} \eta} \nabla \cdot \mathbf{u}}
  -\bar{\rho} g \mathbf{e}_{3},
  \\
  \label{eq:pf_mckenzie_c}
  \nabla \cdot \mathbf{u}
  &= \nabla \cdot \frac{\kappa}{\mu} \nabla \del{p_{\rm f}
  + \rho_{\rm f} g z},
\end{align}
where $\phi$ is porosity, $\mathbf{u}$ is the matrix velocity,
$\boldsymbol{\epsilon}(\mathbf{u}) = (\nabla\mathbf{u} +
(\nabla\mathbf{u})^T)/2$ is the strain rate tensor, $\kappa$ is
permeability, $\mu$ is the melt viscosity, $\eta$ and $\zeta$ are the
shear and bulk viscosity of the matrix, respectively, $g$ is the
constant acceleration due to gravity, $\mathbf{e}_{3}$ is the unit
vector in the $z$-direction (i.e., $\mathbf{e}_{3} = (0, 1)$ when $d =
2$ and $\mathbf{e}_{3} = (0, 0, 1)$ when $d = 3$), $p_{\rm f}$ is the
melt pressure, $\rho_{\rm f}$ and $\rho_{\rm s}$ are the constant melt
and matrix densities, respectively, and $\bar{\rho} = \rho_{\rm f}
\phi + \rho_{\rm s}(1 - \phi)$ is the phase-averaged density. Here we
assume that $\mu$, $\eta$ and $\zeta$ are constants and that $\kappa$
is a function of $\phi$. The magma (fluid) velocity $\mathbf{u}_{\rm
  f}$ can be obtained from $\mathbf{u}$, $\phi$ and $p_{\rm f}$
through:
\begin{equation}
  \label{eq:pf_magma_vel}
  \mathbf{u}_{\rm f} = \mathbf{u}
  - \frac{\kappa}{\phi\mu} \nabla \del{p_{\rm f} + \rho_{\rm f} g z}.
\end{equation}
It will be useful to decompose the melt pressure as $p_{\rm f} = p -
\rho_{\rm s} g z$, where $p$ is the dynamic pressure and $\rho_{\rm s}
g z$ the `lithostatic' pressure. Equations~\eqref{eq:pf_mckenzie_b},
\eqref{eq:pf_mckenzie_c} and~\eqref{eq:pf_magma_vel} may then be
written as
\begin{align}
  \label{eq:p_mckenzie_b}
  -\nabla \cdot 2\eta \boldsymbol{\epsilon}(\mathbf{u}) + \nabla p
  &= \nabla \left(\left(\zeta - \tfrac{2}{3} \eta \right)
  \nabla \cdot \mathbf{u} \right) + g\Delta \rho \phi \mathbf{e}_{3},
  \\
  \label{eq:p_mckenzie_c}
  \nabla \cdot \mathbf{u}
  &= \nabla \cdot \frac{\kappa}{\mu}
  \nabla \left(p - \Delta \rho g z\right),
  \\
  \label{eq:p_magma_vel}
  \mathbf{u}_{\rm f} &= \mathbf{u} - \frac{\kappa}{\phi\mu}
  \nabla \left(p - \Delta \rho g z \right),
\end{align}
where $\Delta \rho = \rho_{\rm s} - \rho_{\rm f}$. Constitutive
relations are given by
\begin{equation}
  \kappa = \kappa_{0} \del{\frac{\phi}{\phi_{0}}}^{n}, \quad
  \zeta = r_{\zeta} \eta,
\end{equation}
where $\phi_{0}$ is the characteristic porosity, $\kappa_{0}$ the
characteristic permeability, $n \ge 1$ is a dimensionless constant and
$r_{\zeta}$ is the ratio between matrix bulk and shear viscosity. We
non-dimensionalise \eqref{eq:pf_mckenzie_a}, \eqref{eq:p_mckenzie_b},
\eqref{eq:p_mckenzie_c} and \eqref{eq:p_magma_vel} using
\begin{equation}
  \mathbf{u} = u_{0} \mathbf{u}^{\prime}, \
  \mathbf{x} = H \mathbf{x}^{\prime}, \
  t = (H/u_0) t^{\prime}, \
  \kappa  = \kappa_{0} \kappa^{\prime}, \
  p = \Delta \rho gH p^{\prime},
\end{equation}
where primed variables are non-dimensional, $u_{0}$ is the velocity
scaling, given by
\begin{equation}
  u_{0} = \frac{\Delta \rho g H^{2}}{2\eta},
\end{equation}
and $H$ is a length scale.  Dropping the prime notation, the McKenzie
equations (\eqref{eq:pf_mckenzie_a}, \eqref{eq:p_mckenzie_b} and
\eqref{eq:p_mckenzie_c}), in non-dimensional form are given by
\begin{align}
  \label{eq:mckenzie_nd_a}
  \partial_{t}\phi - \nabla \cdot \del{(1 - \phi)\mathbf{u}} &= 0,
  \\
  \label{eq:mckenzie_nd_b}
  -\nabla \cdot \boldsymbol{\epsilon}(\mathbf{u}) + \nabla p
  &= \nabla\del{\tfrac{1}{2}\del{r_{\zeta}-\tfrac{2}{3}}
    \nabla \cdot \mathbf{u}} + \phi \mathbf{e}_{3},
  \\
  \label{eq:mckenzie_nd_c}
  \nabla\cdot\mathbf{u}
  &= \frac{2R^{2}}{r_{\zeta} + 4/3}
  \nabla \cdot \del{\del{\frac{\phi}{\phi_{0}}}^{n}
    \del{\nabla p -  \mathbf{e}_{3}}},
\end{align}
where $R = \delta/H$ with $\delta$ the compaction length defined as
\begin{equation}
  \delta = \sqrt{\frac{(r_{\zeta} + 4/3)\kappa_{0} \eta}{\mu}},
\end{equation}
and \eqref{eq:p_magma_vel} becomes
\begin{equation}
  \mathbf{u}_{\rm f} = \mathbf{u} - \frac{2R^{2}}{r_{\zeta} + 4/3}\frac{1}{\phi}
  \del{\frac{\phi}{\phi_{0}}}^{n}
  \del{\nabla p -  \mathbf{e}_{3}}.
\end{equation}

When solving the McKenzie model numerically for time-dependent
simulations, \eqref{eq:mckenzie_nd_a} is usually decoupled from
\eqref{eq:mckenzie_nd_b} and \eqref{eq:mckenzie_nd_c}. Porosity is
updated with \eqref{eq:mckenzie_nd_a} after which the velocity and
pressure are determined by solving \eqref{eq:mckenzie_nd_b} and
\eqref{eq:mckenzie_nd_c}; iteration can be used to better capture the
coupling. The most expensive part of this procedure is solving
\eqref{eq:mckenzie_nd_b} and~\eqref{eq:mckenzie_nd_c}. In this work we
study an optimal solver for equations \eqref{eq:mckenzie_nd_b}
and~\eqref{eq:mckenzie_nd_c} for a given porosity field. We remark
that an alternative to decoupling \eqref{eq:mckenzie_nd_a} from
\eqref{eq:mckenzie_nd_b} and \eqref{eq:mckenzie_nd_c} is to use a
composable linear solver for the full system
\eqref{eq:mckenzie_nd_a}-\eqref{eq:mckenzie_nd_c}, see
\citet{brown:2012}. In this case, our optimal solver may be used as a
preconditioner for part of this composable linear solver.

For the rest of this paper we replace $(r_{\zeta} - 2/3)/2$ by a
constant $\alpha$. Furthermore, we replace
\begin{equation}
  \frac{R^{2}}{\alpha + 1}\del{\frac{\phi}{\phi_{0}}}^{n}
\end{equation}
by a spatially variable function $k(\mathbf{x})$ (independent of
$\alpha$ and $\phi$) and we obtain the problem
\begin{subequations}
  \label{eq:magma}
  \begin{alignat}{1}
    \label{eq:magma_a}
    -\nabla \cdot \boldsymbol{\epsilon}(\mathbf{u}) + \nabla p
    &= \nabla(\alpha\nabla \cdot \mathbf{u}) + \phi  \mathbf{e}_{3},
    \\
    \label{eq:magma_b}
    \nabla \cdot\mathbf{u} &= \nabla \cdot (k(\nabla p -  \mathbf{e}_{3})).
  \end{alignat}
\end{subequations}
For coupled magma/mantle dynamics problems, $\alpha$ may range from
$-1/3$ to approximately $1000$. For this reason we will assume in this
paper that $-1/3 \le \alpha \le 1000$. We also bound $k$: $0 \le k_{*}
\le k(\mathbf{x}) \le k^{*}$ for all $\mathbf{x} \in \Omega$. In the
infinite-dimensional setting, we note that if $k(\mathbf{x}) = 0$
everywhere in $\Omega$, the compaction stress $\nabla (\alpha \nabla
\cdot \mathbf{u})$ vanishes as the velocity field is divergence free
and~\eqref{eq:magma} reduces to the Stokes equations. This will not
generally be the case for a finite element formulation, as will be
discussed in the following section.

On the boundary of the domain, $\partial \Omega$, we impose
\begin{align}
  \label{eq:bc}
    \mathbf{u} &= \mathbf{g},
\\
    \qquad - k(\nabla p - \mathbf{e}_{3}) \cdot \mathbf{n} &= 0,
\end{align}
where $\mathbf{g} : \partial \Omega \rightarrow \mathbb{R}^d$ is given
boundary data satisfying the compatibility condition
\begin{equation}
  0 = \int_{\partial \Omega} \mathbf{g} \cdot \mathbf{n} \dif s.
\end{equation}

\section{Finite element formulation}
\label{s:fem}

In this section we assume, without loss of generality, homogeneous
boundary conditions on~$\mathbf{u}$.

Let $\mathcal{T}_h$ be a triangulation of $\Omega$ with associated
finite element spaces ${\bf X}_{h} \subset
\left(H_{0}^{1}(\Omega)\right)^{d}$ and $M_h \subset H^1(\Omega) \cap
L_{0}^{2}(\Omega)$. The finite element weak formulation for
\eqref{eq:magma} and \eqref{eq:bc} is given by: find $\mathbf{u}_{h},
p_{h} \in \mathbf{X}_{h} \times M_{h}$ such that
\begin{equation}
  \label{eq:fem_weak}
  \mathcal{B}(\mathbf{u}_{h}; p_{h}, \mathbf{v}; q)
  = \int_{\Omega} \phi \mathbf{e}_{3} \cdot  \mathbf{v} \dif x
  - \int_{\Omega} k  \mathbf{e}_{3} \cdot  \nabla q \dif x
  \qquad \forall \mathbf{v}, q \in {\bf X}_{h} \times M_{h},
\end{equation}
where
\begin{equation}
  \label{eq:Bform}
    \mathcal{B}(\mathbf{u}; p, \mathbf{v}; q) = a(\mathbf{u}, \mathbf{v})
    + b(p,\mathbf{v}) + b(q,\mathbf{u}) - c(p,q),
\end{equation}
and
\begin{equation}
  \label{eq:bilinear_parts}
  \begin{split}
    a(\mathbf{u}, \mathbf{v})
    &= \int_{\Omega}
    \boldsymbol{\epsilon}(\mathbf{u}) : \boldsymbol{\epsilon}(\mathbf{v})
    + \alpha(\nabla \cdot \mathbf{u}) (\nabla \cdot \mathbf{v}) \dif x,
    \\
    b(p, \mathbf{v}) &= - \int_{\Omega} p \nabla \cdot \mathbf{v} \dif x,
    \\
    c(p, q) &= \int_{\Omega} k \nabla p \cdot  \nabla q \dif x.
  \end{split}
\end{equation}
\begin{proposition}
  \label{prop:a-stable}
  For $\alpha > -1$, there exists a $c_{\alpha} > 0$ such that
  \begin{equation}
    a(\mathbf{v}, \mathbf{v}) \ge c_{\alpha} \norm{\mathbf{v}}_{1}^{2}
    \quad \forall \ \mathbf{v} \in \del{H^{1}_{0}(\Omega)}^{d}.
  \end{equation}
\end{proposition}
\begin{proof}
  The proposition follows from
  \begin{equation}
    \norm{\nabla \cdot \mathbf{v}}^2
    \le \norm{\boldsymbol{\epsilon}(\mathbf{v})}^2
    \le \norm{\nabla \mathbf{v}}^2
    \quad \forall \ \mathbf{v} \in \del{H^{1}_{0}(\Omega)}^{d},
    \label{eq:div-grad-norm-inequality}
  \end{equation}
  (see Ref.~\citep[Eq.~(3.4)]{Grinevich:2009}) and the application of
  Korn's inequality.
\end{proof}

We will consider finite elements that are inf-sup
stable~\citep{Brezzi:book} in the degenerate limit of $k = 0$, i.e.,
$a(\mathbf{u}, \mathbf{v})$ is coercive (see
Proposition~\ref{prop:a-stable}), $c(p, p) \ge 0 \ \forall p \in
M_{h}$ and for which there exists a constant $c_{1} > 0$ independent
of $h$ such that
\begin{equation}
  \label{eq:lbb}
  \max_{\mathbf{v}_{h} \in \mathbf{X}_{h}}
  \frac{b(q_{h}, \mathbf{v}_{h})}{\norm{\nabla \mathbf{v}_{h}}}
  \ge
  c_{1} \norm{q_h} \qquad \forall q_{h} \in M_{h}.
\end{equation}
In particular, we will use Taylor--Hood ($P^{2}$--$P^{1}$) finite
elements on simplices. We note that while in the infinite-dimensional
setting the Stokes equations are recovered from \eqref{eq:magma} when
$k = 0$, this is not generally the case for the discrete weak
formulation in~\eqref{eq:fem_weak} when~$\alpha \ne 0$. Obtaining the
Stokes limit in the finite element setting when $\alpha \ne 0$
requires the non-trivial property that the divergence of functions in
${\bf X}_h$ lie in the pressure space~$M_{h}$. This is not the case
for Taylor--Hood finite elements.

The discrete system \eqref{eq:fem_weak} can be written in block matrix
form as
\begin{equation}
  \label{eq:matrixForm}
  \begin{bmatrix}
    A & B^T
    \\
    B & -C_{k}
  \end{bmatrix}
  \begin{bmatrix}
    u
    \\
    p
  \end{bmatrix}
  =
  \begin{bmatrix}
    f
    \\
    g
  \end{bmatrix},
\end{equation}
where $u \in \mathbb{R}^{n_{u}}$ and $p \in N^{n_{p}} = \{q\in
\mathbb{R}^{n_{p}}| q\ne 1\}$ are, respectively, the vectors of the
discrete velocity and pressure variables with respect to appropriate
bases for $\mathbf{X}_{h}$ and~$M_{h}$. The space $N^{n_{p}}$ satisfies
the zero mean pressure condition.

For later convenience, we define the negative of the `pressure'
Schur complement~$S$:
\begin{equation}
  S = B A^{-1} B^{T} + C_{k},
  \label{eq:pressure-schur}
\end{equation}
and the scalar pressure mass matrix $Q$ such that
\begin{equation}
  \norm{q_h}^2 = \langle Qq, q \rangle,
  \label{eq:pressure-mass}
\end{equation}
for $q_h \in M_h$ and where $q \in \mathbb{R}^{n_{p}}$ is the vector
of the coefficients associated with the pressure basis and
$\langle \cdot, \cdot \rangle$ denotes the standard Euclidean scalar
product.

The differences between the matrix formulation of the magma/mantle
equations~\eqref{eq:magma} and the Stokes equations lie in the
matrices $A$ and $C_{k}$. In the case of the magma/mantle dynamics,
$A$ includes the discretisation of compaction stresses: a `grad-div'
term weighted by the factor~$\alpha$.  Such `grad-div' terms are known
to be problematic in the context of multigrid methods as the modes
associated with lowest eigenvalues are not well represented on a
coarse grid~\citep{arnold:1997}.  There have been a number of
investigations into this issue for $H({\rm div})$ finite element
problems, e.g.~\citep{Arnold:2000,Kolev:2012}.  The second matrix
which differs from the Stokes discretisation is~$C_{k}$. For
sufficiently large $k$, this term provides Laplace-type pressure
stabilisation for elements that would otherwise be unstable for the
Stokes problem.

\section{Optimal block diagonal preconditioners}
\label{s:precond_stokes}

To model three-dimensional magma/mantle dynamics of subduction,
efficient iterative solvers together with preconditioning techniques
are needed to solve the resulting algebraic systems of equations. The
goal of this section is to introduce and prove optimality of a class
of block diagonal preconditioners for~\eqref{eq:matrixForm}.

To prove optimality of a block preconditioner for the McKenzie
problem, we first present a number of supporting results.

\begin{proposition}
  \label{lem:coercivity_c} The bilinear form $c$
  in~\eqref{eq:bilinear_parts} satisfies
  \begin{equation}
    c(q, q) \ge k_* \norm{\nabla q}^2 \qquad \forall q \in M^h.
    \label{eq:c-coercive}
  \end{equation}
\end{proposition}
\begin{proof}
  This follows directly from
  \begin{equation}
    c(q, q) = \norm{k^{1/2} \nabla q}^2 \ge \norm{k_*^{1/2} \nabla q}^2.
  \end{equation}
\end{proof}

\begin{lemma}
  \label{lem:bounds_for_S_new}
  For the matrices $A$, $B$ and $C_{k}$ given in
  \eqref{eq:matrixForm}, the pressure Schur complement $S$ in
  \eqref{eq:pressure-schur} and the pressure mass matrix $Q$ in
  \eqref{eq:pressure-mass}, for an inf-sup stable formulation
  satisfying~\eqref{eq:lbb}, the following bounds hold
  \begin{equation}
    \label{eq:upperLowerBound}
    0 < c_q
    \le \frac{\langle Sq, q\rangle}{\langle (Q + C_{k})q, q \rangle}
    \le c^q, \quad \forall q \in N^{n_{p}},
  \end{equation}
  where $c^{q}$ is given by
  \begin{equation}
    \label{eq:upper_c}
    c^{q} =
    \begin{cases}
      1/(1 - |\alpha|)    & \mathrm{if}\ -1/3 \le \alpha < 0,\\
      1                   & \mathrm{if}\ \alpha \ge 0,
    \end{cases}
  \end{equation}
  and $c_{q}$ by
  \begin{equation}
    \label{eq:lower_c}
    c_{q} = \min \del{\frac{c_{1}^{2} + c_{P} k_{*}(1
        + |\alpha|)}{(1 + |\alpha|)(1 + c_{P} k_{*})}, \ 1},
  \end{equation}
  where $c_{1}$ is the inf-sup constant and $c_{P}$ the Poincar\'e
  constant.
\end{lemma}
\begin{proof}
  Since $A$ is symmetric and positive definite, and from the
  definition of $S$
  \begin{equation}
    \begin{split}
     \langle Sq, q \rangle &= \langle A^{-1} B^{T} q, B^{T} q \rangle
      + \langle C_{k} q, q \rangle
      \\
      &= \sup_{v \in \mathbb{R}^{n_{u}}}
      \frac{\langle v, B^{T} q \rangle^{2}}{\langle A v, v \rangle}
      + \langle C_{k} q, q \rangle,
    \end{split}
  \end{equation}
  for all $q \in N^{n_{p}}$. From
  the definition of matrices $A, B, C_{k}$ and~$Q$ it then follows
  that
  \begin{equation}
    \label{eq:Sqq}
      \langle Sq, q \rangle
      = \sup_{\mathbf{v}_{h} \in \mathbf{X}_{h}}
      \frac{(q_{h}, \nabla \cdot \mathbf{v}_{h})^{2}}
              {\norm{\boldsymbol{\epsilon}(\mathbf{v}_{h})}^{2}
        + \alpha \norm{\nabla\cdot \mathbf{v}_{h}}^{2}}
      + (k \nabla q_{h}, \nabla q_{h}).
  \end{equation}

  Using~\eqref{eq:div-grad-norm-inequality} and the Cauchy--Schwarz
  inequality,
  \begin{equation}
    (q_{h}, \nabla \cdot \mathbf{v}_h)^{2}
    \le \norm{q_{h}}^{2} \norm{\boldsymbol{\epsilon}(\mathbf{v}_{h})}^{2}.
  \end{equation}
  For $-1/3 \le \alpha < 0$,
  \begin{equation}
    \begin{split}
    \norm{\boldsymbol{\epsilon}(\mathbf{v}_h)}^2
    &= \frac{1}{1 + \alpha} \del{\norm{\boldsymbol{\epsilon}(\mathbf{v}_h)}^2
    + \alpha\norm{\boldsymbol{\epsilon}(\mathbf{v}_h)}^2}
    \\
    &\le \frac{1}{1 + \alpha}\del{\norm{\boldsymbol{\epsilon}(\mathbf{v}_h)}^2
    + \alpha\norm{\nabla \cdot \mathbf{v}_h}^2},
    \end{split}
  \end{equation}
  and for $\alpha \ge 0$,
  \begin{equation}
    \norm{\boldsymbol{\epsilon}(\mathbf{v}_{h})}^{2}
    \le \norm{\boldsymbol{\epsilon}(\mathbf{v}_{h})}^{2}
    + \alpha\norm{\nabla \cdot \mathbf{v}_{h}}^{2}.
  \end{equation}
  Hence,
  \begin{equation}
    \label{eq:uboundqdivu}
    (q_{h}, \nabla \cdot \mathbf{v}_{h})^{2}
    \le c^q \norm{q_{h}}^{2} \left(\norm{\boldsymbol{\epsilon}(\mathbf{v}_h)}^2
    + \alpha\norm{\nabla \cdot \mathbf{v}_h}^2\right),
  \end{equation}
  where
  \begin{equation}
    c^q =
    \begin{cases}
      1/(1 - |\alpha|)  & \mathrm{if} \ -1/3 \le \alpha < 0,\\
      1                 & \mathrm{if} \ \alpha \ge 0.
    \end{cases}
  \end{equation}
  Combining \eqref{eq:Sqq} and~\eqref{eq:uboundqdivu},
  \begin{equation}
    \label{eq:upperbound_Sqq}
      \langle Sq, q\rangle \le c^q \norm{q_h}^2 + (k \nabla q_h, \nabla q_h)
      = c^q \langle Qq, q\rangle + \langle C_{k}q, q \rangle
      \le c^q \langle (Q + C_{k})q, q\rangle.
  \end{equation}
  This proves the upper bound in~\eqref{eq:upperLowerBound}.

  Next we determine the lower bound. Using
  \eqref{eq:div-grad-norm-inequality} and the inf-sup
  condition~\eqref{eq:lbb},
  \begin{equation}
    \label{eq:Approx_ge_zero}
    \begin{split}
      \max_{\mathbf{v}_{h} \in \mathbf{X}_{h}} \frac{(q_h, \nabla
        \cdot
        \mathbf{v}_h)^2}{\norm{\boldsymbol{\epsilon}(\mathbf{v}_h)}^2
        + \alpha\norm{\nabla \cdot \mathbf{v}_h}^2} & \ge
      \max_{\mathbf{v}_{h} \in \mathbf{X}_{h}} \frac{(q_h, \nabla
        \cdot \mathbf{v}_h)^{2}}{(1 + |\alpha|) \norm{\nabla
          \mathbf{v}_h}^{2}} \\ &\ge \frac{c_{1}^{2}}{1+|\alpha|}
      \norm{q_{h}}^{2},
    \end{split}
  \end{equation}
  which leads to
  \begin{equation}
    \label{eq:Sq1}
    \langle Sq, q\rangle \ge \frac{c_{1}^{2}}{ 1 +|\alpha|}
       \langle Qq, q  \rangle
       + \langle C_{k}q, q \rangle.
  \end{equation}
  Using Proposition~\ref{lem:coercivity_c} and the Poincar\'e
  inequality,
  \begin{equation}
    \label{eq:Cq1}
    \begin{split}
      \langle C_{k}q, q \rangle
      &= (1 - \xi) c(q_h, q_h)
      + \xi \norm{k^{1/2} \nabla q_{h}}^{2}
      \\
      &\ge (1-\xi) c(q_h, q_h)
      + \xi c_{P} k_{*} \norm{q_h}^2
      \\
      &= (1-\xi) \langle C_{k}q, q \rangle
      + \xi c_{P} k_{*} \langle Qq, q\rangle,
    \end{split}
  \end{equation}
  for any $\xi \in [0, 1]$. Combining~\eqref{eq:Sq1}
  and~\eqref{eq:Cq1},
  \begin{equation}
      \langle Sq, q\rangle
      \ge
      \del{\frac{c_{1}^2}{1 + |\alpha|}
        + \xi c_{P} k_{*}} \langle Qq, q \rangle
      + (1 - \xi) \langle C_{k}q, q \rangle,
  \end{equation}
  and setting $\xi = (1 - c_{1}^{2}/(1 +
  |\alpha|))/(1 + c_{P}k_{*})$ in the case that $c_{1}^{2}/(1 +
  |\alpha|) \le 1$, and otherwise setting $\xi = 0$,
  \begin{equation}
    \label{eq:Sq2}
        \langle Sq, q\rangle
        \ge
         \min \del{\frac{c_{1}^{2} + c_{P} k_{*}(1
        + |\alpha|)}{(1 + |\alpha|)(1 + c_{P} k_{*})}, \ 1}
         \langle (Q + C_{k})q, q \rangle,
  \end{equation}
  from which $c_q$ is deduced.
\end{proof}

For the discretisation of the Stokes equations, it was shown that the
pressure mass-matrix is spectrally equivalent to the Schur
complement~\citep{Silvester:1994}. This is recovered from
Lemma~\ref{lem:bounds_for_S_new} when $k = 0$ everywhere and~$\alpha =
0$.

\begin{lemma}
  \label{lem:bounds_for_S_Q}
  For the matrices $A$, $B$ and $C_{k}$ in \eqref{eq:matrixForm},
  $S$ in \eqref{eq:pressure-schur} and the pressure mass matrix $Q$ in
  \eqref{eq:pressure-mass}, if the inf-sup condition in \eqref{eq:lbb}
  is satisfied, then
  \begin{equation}
    \label{eq:upperBound_S_Q}
    \frac{\langle (B^T (Q + C_{k})^{-1} B v, v\rangle}{\langle Av, v \rangle}
    \le c^q \quad \forall v \in \mathbb{R}^{n_{u}},
  \end{equation}
  where $c^q$ is the constant from in~\eqref{eq:upper_c}.
\end{lemma}
\begin{proof}
  From Lemma~\ref{lem:bounds_for_S_new}, symmetry of $A$ and positive
  semi-definiteness of~$C$,
  \begin{equation}
    \frac{q^T B A^{-1} B^T q}{q^T \del{Q + C_{k}} q}
    \le \frac{q^T \del{B A^{-1} B^T + C_{k}} q}{q^T \del{Q + C_{k}} q}
    \le c^q \quad \forall q \in N^{n_{p}}.
  \end{equation}
  Inserting $q \gets (Q + C_{k})^{1/2} q$,
  \begin{equation}
    \frac{q^T (Q + C_{k})^{-1/2} B A^{-1} B^T (Q + C_{k})^{-1/2} q}{q^T q}
    \le c^q \quad \forall q \in N^{n_{p}}.
  \end{equation}
  Defining $H = (Q + C_{k})^{-1/2} B A^{-1} B^T(Q +
  C_{k})^{-1/2}$ and denoting the maximum eigenvalue of $H$ by
  $\lambda_{\max}$ and associated eigenvector $x$, since $H$ is
  symmetric it follows that $\lambda_{\rm max} \ge v^{T} H v/(v^{T} v)
  \ \forall v \in \mathbb{R}^{n}$ and $\lambda_{\rm max} = x^{T} H
  x/(x^{T} x)$. Hence, $\lambda_{\max} \le c^q$, and
  \begin{equation}
    (Q + C_{k})^{-1/2} B A^{-1} B^T(Q + C_{k})^{-1/2} x = \lambda_{\max} x,
  \end{equation}
  and pre-multiplying both sides by $ A^{-1/2} B^T(Q + C_{k})^{-1/2}$,
  \begin{multline}
    A^{-1/2} B^T(Q + C_{k})^{-1/2}(Q
    + C_{k})^{-1/2} B A^{-1/2} A^{-1/2} B^T (Q+C_{k})^{-1/2} x
    \\
    = \lambda_{\max} A^{-1/2} B^T (Q + C_{k})^{-1/2} x.
  \end{multline}
  Letting $v = A^{-1/2} B^T (Q + C_{k})^{-1/2} x$, the above becomes
  \begin{equation}
    A^{-1/2} B^T (Q + C_{k})^{-1} B A^{-1/2} v = \lambda_{\max} v,
  \end{equation}
  and it follows from $\lambda_{\max} \le c^q$ that
  \begin{equation}
    \frac{v^T A^{-1/2} B^T (Q + C_{k})^{-1} B A^{-1/2} v}{v^T v} \le c^q
    \quad \forall v \in \mathbb{R}^{n_{u}},
  \end{equation}
  or, taking $v \leftarrow A^{-1/2} v$,
  \begin{equation}
    \frac{v^T B^T(Q + C_{k})^{-1} B v}{v^T A v} \le c^q
    \quad \forall v \in \mathbb{R}^{n_{u}},
  \end{equation}
  and the Lemma follows.
\end{proof}

We now consider diagonal block preconditioners
for~\eqref{eq:matrixForm} of the form
\begin{equation}
  \label{eq:precon_ideal}
  \mathcal{P}
  =
  \begin{bmatrix}
    P & 0 \\
    0 & T
  \end{bmatrix},\qquad
  P\in\mathbb{R}^{n_u\times n_u},\quad
  T\in\mathbb{R}^{n_p\times n_p}.
\end{equation}
We assume that $P$ and $T$ are symmetric and positive-definite, and
that they satisfy
\begin{equation}
  \label{eq:boundsAPQT}
  \delta_{AP} \le \frac{\langle Av, v\rangle}{\langle Pv, v\rangle}
  \le \delta^{AP} \ \forall v\in\mathbb{R}^{n_u},
  \quad
  \delta_{QT} \le
  \frac{\langle(Q + C_{k})q, q \rangle}{\langle Tq, q\rangle}
  \le \delta^{QT} \ \forall q\in N^{n_p},
\end{equation}
where $\delta_{AP}$, $\delta^{AP}$, $\delta_{QT}$ and $\delta^{QT}$
are independent of $h$, but may depend on model parameters.

The discrete system in~\eqref{eq:matrixForm} is indefinite, and hence
has both positive and negative eigenvalues.  The speed of convergence
of the MINRES Krylov method for the preconditioned system
\begin{equation}
  \label{eq:precondsys}
  \begin{bmatrix}
    P & 0 \\
    0 & T
  \end{bmatrix}^{-1}
  \begin{bmatrix}
    A & B^T \\
    B & -C_{k}
  \end{bmatrix}
  \begin{bmatrix}
    u \\ p
  \end{bmatrix}=
  \begin{bmatrix}
    P & 0 \\
    0 & T
  \end{bmatrix}^{-1}
  \begin{bmatrix}
    f \\ g
  \end{bmatrix},
\end{equation}
depends on how tightly the positive and negative eigenvalues of the
generalised eigenvalue problem
\begin{equation}
  \label{eq:geneigenvalue}
  \begin{bmatrix}
    A & B^T \\
    B & -C_{k}
  \end{bmatrix}
  \begin{bmatrix}
    v \\ q
  \end{bmatrix}=
  \lambda
  \begin{bmatrix}
    P & 0 \\
    0 & T
  \end{bmatrix}
  \begin{bmatrix}
    v \\ q
  \end{bmatrix},
\end{equation}
are clustered~\citep[Section~6.2]{Elman:book}.  Our aim now is to
develop bounds on the eigenvalues in~\eqref{eq:geneigenvalue} that are
independent of the mesh parameter~$h$.

\begin{theorem}
  \label{lem:boundsPu}
  Let $c_q$ and $c^q$ be the constants in
  Lemma~\ref{lem:bounds_for_S_new}, and the matrices $A$, $B$ and
  $C_{k}$ be those given in \eqref{eq:matrixForm}, $S$ be the pressure
  Schur complement in \eqref{eq:pressure-schur} and $Q$ the pressure
  mass matrix in~\eqref{eq:pressure-mass}. If $P$ and $T$ satisfy
  \eqref{eq:boundsAPQT}, all eigenvalues $\lambda < 0$ of
  \eqref{eq:geneigenvalue} satisfy
  \begin{equation}
    \label{eq:negEigs_diag}
    - c^q \delta^{QT}
    \le
    \lambda
    \le
    \tfrac{1}{2}\del{\delta_{AP} - \sqrt{\delta_{AP}^2
        + 4 c_q \delta_{QT} \delta_{AP}}},
  \end{equation}
  and eigenvalues $\lambda > 0$ of \eqref{eq:geneigenvalue} satisfy
  \begin{equation}
    \label{eq:posEigs_diag}
    \delta_{AP}
    \le
    \lambda
    \le \delta^{AP} + c^q \delta^{QT}.
  \end{equation}
\end{theorem}
\begin{proof}
  Lemmas~\ref{lem:bounds_for_S_new} and~\ref{lem:bounds_for_S_Q}
  provide the bounds
  \begin{equation}
    c_q
    \le \frac{\langle Sq, q\rangle}{\langle (Q + C_{k})q, q \rangle}
    \le c^q,
    \quad
    \frac{\langle (B^T (Q + C_{k})^{-1}B v, v\rangle}{\langle Av, v \rangle}
    \le c^q,
  \end{equation}
  for all $q \in N^{n_p}$ and $\forall v \in
  \mathbb{R}^{n_{u}}$. Using these bounds together with the bounds
  given in \eqref{eq:boundsAPQT}, the result follows directly by
  following the proof of Theorem 6.6 in~\citet{Elman:book}, or more
  generally~\citet{Pestana:2013}.
\end{proof}

The main result of this section, Theorem~\ref{lem:boundsPu}, states
that the eigenvalues of the generalised eigenvalue
problem~\eqref{eq:geneigenvalue} are independent of the problem size.
From Theorem~\ref{lem:boundsPu} we see that
\begin{equation}
  \lambda \in \sbr{-c^q \delta^{QT}, \ \frac{1}{2} \del{\delta_{AP}
    - \sqrt{\delta_{AP}^{2} + 4c_q \delta_{QT} \delta_{AP}}}}
  \bigcup \sbr{\delta_{AP}, \ \delta^{AP} + c^q \delta^{QT}},
\label{eq:lambda-interval}
\end{equation}
in which all constants are independent of the problem size
(independent of $h$). This tells us that if we can find a $P$ and $T$
that are spectrally equivalent to $A$ and $Q + C_{k}$,
respectively, then an iterative method with preconditioner
\eqref{eq:precon_ideal} will be optimal for~\eqref{eq:matrixForm}.

The interval in~\eqref{eq:lambda-interval} shows the dependence of the
eigenvalues on~$\alpha$ and~$k$. The upper and lower bounds on the
positive eigenvalues are well behaved, as is the lower bound on the
negative eigenvalues, for all $\alpha$ and $k$. It is only when $c_{q}
\ll 1$ that the upper bound on the negative eigenvalues tends to
zero. If this is the case, the rate of convergence of the iterative
method may slow. From \eqref{eq:lower_c}, we see that $c_{q} \ll 1$
only if $\alpha \gg 1$ and, at the same time,~$k_{*} \ll 1$.

\section{Preconditioner construction}
\label{s:pcconstruction}

Implementation of the proposed preconditioner requires the provision
of symmetric, positive definite matrices $P$ and $T$ that
satisfy~\eqref{eq:boundsAPQT}.  Obvious candidates are $P = A$ and $T
= Q + C_k$, with a direct solver used to compute the action of
$P^{-1}$ and~$T^{-1}$. We will use this for small problems in the
following section to study the performance of the block
preconditioning; the application of a direct solver is not practical,
however, when $P$ and $T$ are large, in which case we advocate the use
of multigrid approximations of the inverse.

To provide more general guidance, we first reproduce the following
Lemma from \citet[Lemma 6.2]{Elman:book}.
\begin{lemma}
  \label{lem:rho_bounds}
  If $u$ is the solution to the system $A u = f$ and
  \begin{equation}
    u_{i + 1} = (I - P^{-1} A) u_{i} + P^{-1} f,
  \end{equation}
  then if the iteration error satisfies $\langle A (u - u_{i + 1}), u
  - u_{i + 1} \rangle \le \rho \langle A (u - u_{i}), u - u_{i}
  \rangle$, with $\rho < 1$,
  \begin{equation}
    1 - \rho
    \le
    \frac{\langle Av, v\rangle}{\langle Pv, v\rangle}
    \le
    1 + \rho \quad \forall v.
    \label{eq:rho-convergence}
  \end{equation}
\end{lemma}
\begin{proof}
  See \citet[proof of Lemma 6.2]{Elman:book}.
\end{proof}

Lemma~\ref{lem:rho_bounds} implies that a solver that is optimal for
$A u = f$ will satisfy~\eqref{eq:boundsAPQT}, and is therefore a
candidate for $P$, and likewise for~$T$. The obvious candidates for
$P$ and $T$ are multigrid preconditioners applied to $A$ and $Q +
C_k$, respectively. However, as we will show by example in
Section~\ref{s:numsim}, as $\alpha$ increases, and therefore the
compaction stresses (a `grad-div' term) become more important,
multigrid for $P$ becomes less effective as a preconditioner.  More
effective treatment of the large $\alpha$ case is the subject of
ongoing investigations.

\section{Numerical simulations}
\label{s:numsim}

In this section we verify the analysis results through numerical
examples. In all test cases we use $P^{2}$--$P^{1}$ Taylor--Hood
finite elements on simplices. The numerical examples deliberately
address points of practical interest such as spatial variations in the
parameter $k$, a wide range of values for $\alpha$ and large problem
sizes on unstructured grids of subduction zone-like geometries.

We consider two preconditioners. For the first, we take $P = A$ and $T
= Q + C_{k}$ in~\eqref{eq:precon_ideal} and apply a direct solver to
compute the action of the inverses.  This preconditioner will be
referred to as the `LU' preconditioner.  For the second, we use
$P^{-1} = A^{\rm AMG}$ and $T^{-1} =(Q + C_{k})^{\rm AMG}$, where we use
$(\cdot)^{\rm AMG}$ to denote the use of algebraic multigrid to
approximate the inverse of~$(\cdot)$.  This preconditioner will be
referred to as the `AMG' preconditioner.  The LU preconditioner is
introduced as a reference preconditioner to which the AMG
preconditioner can be compared. The LU preconditioner is not suitable
for large scale problems. Note that we never construct the inverse of
$P$ or $T$, but that we just use the action of the inverse.

All tests use the MINRES method, and the solver is terminated once a
relative true residual of $10^{-8}$ is reached.  For multigrid
approximations of $P^{-1}$, smoothed aggregation algebraic multigrid
is used via the library ML~\citep{gee:2006}. For multigrid
approximations of $T^{-1}$, classical algebraic multigrid is used via
the library BoomerAMG~\citep{henson:2002}.  Unless otherwise stated,
we use multigrid V-cycles, with two applications of Chebyshev with
Jacobi smoothing on each level (pre and post) in the case of smoothed
aggregation, and symmetric Gauss--Seidel for the classical algebraic
multigrid. The computer code is developed using the finite element
library DOLFIN~\citep{logg:2010}, with block preconditioner support
from PETSc~\citep{brown:2012} to construct the preconditioners.  The
computer code to reproduce all examples is freely available in the
supporting material~\citep{supporting-material}.

\subsection{Verification of optimality}
\label{ss:tc1}

In this test case we verify optimality of the block preconditioned
MINRES scheme by observing the convergence of the solver for varying
$h$, $\alpha$, $k^{*}$ and $k_{*}$. We solve \eqref{eq:magma} and
\eqref{eq:bc} on the unit square domain $\Omega = (0, 1)^{2}$ using a
regular mesh of triangular cells. For the permeability, we consider
\begin{multline}
  k =
  \frac{k^{*} - k_{*}}{4\tanh(5)}
  \left(\tanh(10 x - 5)
    + \tanh(10z - 5) \right.
\\
\left.
    +  \frac{2(k^{*} - k_{*})
      - 2 \tanh(5)(k_{*}
      + k^{*})}{k_{*} - k^{*}}
  + 2\right).
\end{multline}
We ignore body forces but add a source term $\mathbf{f}$ to the right
hand side of~\eqref{eq:magma_a}. The Dirichlet boundary condition
$\mathbf{g}$ and the source term $\mathbf{f}$ are constructed such
that the exact solution pressure $p$ and velocity $\mathbf{u}$ are:
\begin{align}
 p     &= -\cos(4 \pi x)\cos(2 \pi z),
 \\
  u_{x} &= k \partial_{x} p + \sin(\pi x)\sin(2\pi z) + 2,
  \\
  u_{z} &= k \partial_{z} p +  \frac{1}{2}\cos(\pi x)\cos(2\pi z) + 2.
\end{align}

Table~\ref{tab:tc1_minres_mu} shows the number of iterations the
MINRES method required to converge using the LU and AMG
preconditioners with $k_{*} = 0.5$ and $k^{*} = 1.5$, when varying
$\alpha$ from $-1/3$ to $1000$. We clearly see that the LU
preconditioner is optimal (the iteration count is independent of the
problem size), as predicted by the analysis (see
Theorem~\ref{lem:boundsPu}). Using the AMG preconditioner, there is a
very slight dependence on the problem size.  The results in
Table~\ref{tab:tc1_minres_mu} indicate that the LU preconditioner is
uniform with respect to~$\alpha$. Theorem~\ref{lem:boundsPu} indicates
a possible dependence on~$\alpha$ through the constant~$c_q$. However,
for $\alpha$ sufficiently small or sufficiently large, the dependence
of $c_{q}$ on $\alpha$ becomes negligible, and $\alpha$ has only a
small impact on the iteration count. The AMG preconditioner, on the
other hand, shows a strong dependence on~$\alpha$.  The issue with the
`grad-div' for multigrid solvers was discussed in Section~\ref{s:fem},
and is manifest in Table~\ref{tab:tc1_minres_mu}.  It has been
observed in tests that the effectiveness of a multigrid preconditioned
solver for the operator $A$ deteriorates with increasing~$\alpha$.
This is manifest in an increasing $\rho$ in~\eqref{eq:rho-convergence}
for increasing~$\alpha$.

\begin{table}
  \caption{Number of iterations for the LU and AMG preconditioned
    MINRES for the unit square test with different levels of mesh
    refinement and for different values of~$\alpha$. The number of
    degrees-of-freedom is denoted by~$N$.  For the $\alpha = 1000$
    case, four applications of a Chebyshev smoother, with one
    symmetric Gauss-Seidel iteration for each application, was used.}
\begin{center}
\begin{tabular}{c|cc|cc|cc|cc|cc}
& \multicolumn{2}{c|}{$\alpha=-\tfrac{1}{3}$} & \multicolumn{2}{c|}{$\alpha=0$} & \multicolumn{2}{c|}{$\alpha=1$} & \multicolumn{2}{c|}{$\alpha=10$} & \multicolumn{2}{c}{$\alpha=1000$}\\
$N$      & LU & AMG   & LU & AMG     & LU & AMG     & LU & AMG    & LU & AMG$^{*}$ \\
\hline
  9,539  & 9  & 29    & 9  & 30      & 9  & 35      & 8 & 67      & 7  & 202 \\
 37,507  & 9  & 33    & 9  & 36      & 9  & 40      & 8 & 80      & 6  & 283 \\
148,739  & 8  & 39    & 8  & 40      & 9  & 47      & 7 & 96      & 6  & 366 \\
592,387  & 8  & 42    & 8  & 44      & 7  & 52      & 7 & 106     & 6  & 432
\end{tabular}
\label{tab:tc1_minres_mu}
\end{center}
\end{table}

Results for the case of large spatial variations in permeability $k$
are presented in Tables~\ref{tab:tc1_minres_k_a}
and~\ref{tab:tc1_minres_k_b} for the cases~$\alpha = 1$ and~$\alpha =
100$, respectively. A dependence of the iteration count on the
permeability is observed. The smaller $k^{*}$, the larger the
iteration counts for both the AMG and the LU preconditioners. We also
observe that for a given $k^{*}$ there is little influence of $k_{*}$
on the iteration count. Comparing the results in
Tables~\ref{tab:tc1_minres_k_a} and~\ref{tab:tc1_minres_k_b} we see
that the LU preconditioner shows no dependence on $\alpha$. For the AMG
preconditioner the iteration count increases as $\alpha$ increases
from 1 to 100.

\begin{table}
  \caption{Number of iterations to reach a relative tolerance of
    $10^{-8}$ using preconditioned MINRES for the unit square test
    with varying levels of mesh refinement and varying $(k_{*},
    k^{*})$ pairs for $\alpha=1$. The number of degrees of freedom is
    denoted by~$N$.}
\begin{center}
\begin{tabular}{c|cc|cc|cc|cc}
$k^*=10^{-4}$ & \multicolumn{2}{c|}{$k_{*} = 0$}
       & \multicolumn{2}{c|}{$k_{*} = 10^{-8}$}
       & \multicolumn{2}{c}{$k_{*} = 10^{-6}$}
       & \multicolumn{2}{c}{$k_{*} = 5\cdot 10^{-5}$}\\
$N$     & LU & AMG  & LU & AMG  & LU & AMG  & LU & AMG \\
\hline
 9,539  & 32 & 88   & 32 & 88   & 32 & 88   & 32 & 80\\
 37,507 & 35 & 108  & 35 & 108  & 35 & 108  & 35 & 97\\
148,739 & 38 & 130  & 37 & 130  & 38 & 127  & 33 & 111\\
592,387 & 36 & 143  & 36 & 143  & 35 & 135  & 33 & 122\\
\multicolumn{9}{c}{}\\
$k^*=1$ & \multicolumn{2}{c|}{$k_{*} = 0$}
       & \multicolumn{2}{c|}{$k_{*} = 0.1$}
       & \multicolumn{2}{c}{$k_{*} = 0.5$}
       & \multicolumn{2}{c}{$k_{*} = 0.9$}\\
$N$      & LU & AMG   & LU & AMG  & LU & AMG  & LU & AMG\\
\hline
 9,539   & 27 & 67    & 10 & 37   & 9 & 36    & 9 & 36\\
 37,507  & 28 & 78    & 10 & 44   & 9 & 42    & 9 & 42\\
148,739  & 28 & 93    & 10 & 50   & 9 & 48    & 7 & 47\\
592,387  & 27 & 101   & 10 & 54   & 9 & 52    & 7 & 52\\
\multicolumn{9}{c}{}\\
$k^*=1000$ & \multicolumn{2}{c|}{$k_{*} = 0$}
       & \multicolumn{2}{c|}{$k_{*} = 1$}
       & \multicolumn{2}{c}{$k_{*} = 10$}
       & \multicolumn{2}{c}{$k_{*} = 100$}\\
$N$     & LU & AMG  & LU & AMG  & LU & AMG & LU & AMG\\
\hline
 9,539  & 3 & 24    & 3 & 26    & 3 & 24   & 3 & 24\\
 37,507 & 3 & 27    & 3 & 27    & 3 & 27   & 3 & 30\\
148,739 & 3 & 34    & 3 & 33    & 3 & 34   & 3 & 33\\
592,387 & 3 & 37    & 3 & 37    & 3 & 37   & 3 & 40\\
\multicolumn{9}{c}{}\\
$k^*=10^{8}$ & \multicolumn{2}{c|}{$k_{*} = 0$}
       & \multicolumn{2}{c|}{$k_{*} = 1$}
       & \multicolumn{2}{c}{$k_{*} = 10^{3}$}
       & \multicolumn{2}{c}{$k_{*} = 10^{6}$}\\
$N$      & LU & AMG  & LU & AMG  & LU & AMG & LU & AMG\\
\hline
 9,539   & 1  & 15   & 1  & 15   & 1  & 15   & 1  & 15\\
 37,507  & 2  & 18   & 2  & 18   & 2  & 18   & 2  & 18\\
148,739  & 2  & 21   & 2  & 21   & 2  & 21   & 2  & 21\\
592,387  & 2  & 21   & 2  & 21   & 2  & 21   & 2  & 21
\end{tabular}
\label{tab:tc1_minres_k_a}
\end{center}
\end{table}

\begin{table}
  \caption{Number of iterations to reach a relative tolerance of
    $10^{-8}$ using preconditioned MINRES for the unit square test
    with varying levels of mesh refinement and varying $(k_{*},
    k^{*})$ pairs for $\alpha=100$. The number of degrees of freedom is
    denoted by~$N$.}
\begin{center}
\begin{tabular}{c|cc|cc|cc|cc}
$k^*=10^{-4}$ & \multicolumn{2}{c|}{$k_{*} = 0$}
       & \multicolumn{2}{c|}{$k_{*} = 10^{-8}$}
       & \multicolumn{2}{c}{$k_{*} = 10^{-6}$}
       & \multicolumn{2}{c}{$k_{*} = 5\cdot 10^{-5}$}\\
$N$     & LU  & AMG   & LU  & AMG   & LU  & AMG   & LU & AMG \\
\hline
 9,539  & 67  & 1605  & 67  & 1598  & 66  & 1557  & 58  & 1385\\
 37,507 & 75  & 1922  & 75  & 1922  & 71  & 1909  & 62  & 1730\\
148,739 & 76  & 2179  & 76  & 2177  & 72  & 2146  & 59  & 1972\\
592,387 & 73  & 2356  & 73  & 2356  & 68  & 2311  & 59  & 2156\\
\multicolumn{9}{c}{}\\
$k^*=1$ & \multicolumn{2}{c|}{$k_{*} = 0$}
       & \multicolumn{2}{c|}{$k_{*} = 0.1$}
       & \multicolumn{2}{c}{$k_{*} = 0.5$}
       & \multicolumn{2}{c}{$k_{*} = 0.9$}\\
$N$     & LU & AMG  & LU & AMG  & LU & AMG  & LU & AMG\\
\hline
 9,539  & 28 & 350  & 9 & 179   & 8  & 171  & 7 & 169\\
 37,507 & 28 & 445  & 9 & 212   & 8  & 205  & 8 & 202\\
148,739 & 28 & 545  & 9 & 247   & 8  & 236  & 8 & 234\\
592,387 & 28 & 597  & 9 & 271   & 8  & 265  & 8 & 265\\
\multicolumn{9}{c}{}\\
$k^*=1000$ & \multicolumn{2}{c|}{$k_{*} = 0$}
       & \multicolumn{2}{c|}{$k_{*} = 1$}
       & \multicolumn{2}{c}{$k_{*} = 10$}
       & \multicolumn{2}{c}{$k_{*} = 100$}\\
$N$      & LU & AMG  & LU & AMG  & LU & AMG  & LU & AMG\\
\hline
 9,539  & 3   & 75   & 3  & 75   & 3  & 75   & 3  & 75\\
 37,507 & 3   & 94   & 3  & 94   & 3  & 94   & 3  & 94\\
148,739 & 3   & 116  & 3  & 116  & 3  & 116  & 3  & 116\\
592,387 & 3   & 139  & 3  & 139  & 3  & 139  & 3  & 139\\
\multicolumn{9}{c}{}\\
$k^*=10^{8}$ & \multicolumn{2}{c|}{$k_{*} = 0$}
       & \multicolumn{2}{c|}{$k_{*} = 1$}
       & \multicolumn{2}{c}{$k_{*} = 10^{3}$}
       & \multicolumn{2}{c}{$k_{*} = 10^{6}$}\\
$N$     & LU & AMG  & LU & AMG  & LU & AMG & LU & AMG\\
\hline
 9,539  & 1  & 11   & 1  & 11   & 1  & 11  & 1  & 11\\
 37,507 & 1  & 13   & 1  & 13   & 1  & 13  & 1  & 13\\
148,739 & 1  & 20   & 1  & 20   & 1  & 20  & 1  & 20\\
592,387 & 1  & 23   & 1  & 23   & 1  & 23  & 1  & 23
\end{tabular}
\label{tab:tc1_minres_k_b}
\end{center}
\end{table}

\subsection{A magma dynamics problem in two dimensions}
\label{ss:tc2}

In this test case we solve \eqref{eq:magma} and \eqref{eq:bc} on a
domain $\Omega$, depicted in Figure~\ref{fig:2Dsubduction}, using
unstructured meshes with triangular cells.  We take $L_{x}^{t} = 1.5$,
$L_{x}^{b} = 0.5$ and $L_{z} = 1$. We set the permeability as
$k = 0.9(1 + \tanh(-2r))$ with $r = \sqrt{x^{2} + z^{2}}$
and the porosity~$\phi = 0.01$.
\begin{figure}
  \centering
  \includegraphics[width=0.45\textwidth]{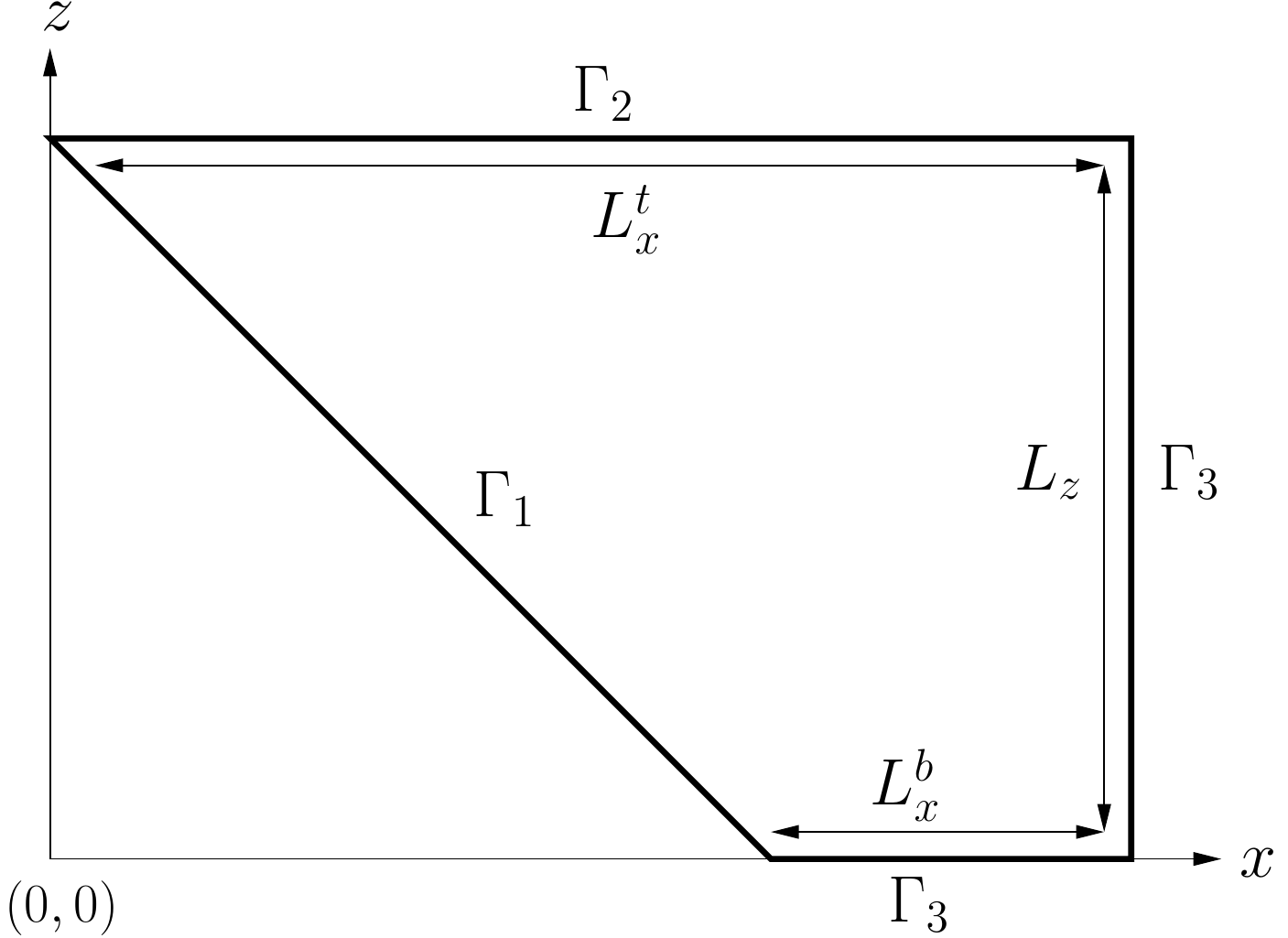}
  \caption{Description of the wedge geometry for a two-dimensional
    subduction zone.}
  \label{fig:2Dsubduction}
\end{figure}

We consider two test cases for this geometry. The first test problem
we denote as the \emph{analytical corner flow} test problem and the
second as the \emph{traction-free} test problem. In both problems we
prescribe the following conditions: $\mathbf{u} =
\mathbf{u}_{\text{slab}} = (1, -1)/\sqrt{2}$ on $\Gamma_1$,
$\mathbf{u} = \mathbf{0}$ on $\Gamma_2$ and $-k\del{\nabla p -
  \mathbf{e}_{3}} \cdot \mathbf{n} = 0$ on $\partial \Omega$.

\subsubsection{Analytic corner flow}
\label{ss:corner}

For the analytical corner flow problem we prescribe $\mathbf{u} =
\mathbf{u}_{\rm corner} = (u_{x}, u_{z})$ on $\Gamma_{3}$, which is
the analytic expression for
corner-flow~\citep[Section~4.8]{Batchelor:book}. The corner-flow
velocity components $u_{x}$ and $u_{z}$ are given by
\begin{equation}
  u_{x} = \cos(\theta) u_{r}  + \sin(\theta)u_{\theta},
  \quad
  u_{z} = -\sin(\theta) u_{r} + \cos(\theta)u_{\theta},
\end{equation}
where $\theta = -\arctan(\tilde{z}/x)$, $\tilde{z} = z - 1$ and
\begin{equation}
  u_r = C \theta \sin(\theta) + D(\sin(\theta) + \theta\cos(\theta)),
  \quad
  u_{\theta} = C(\sin(\theta) - \theta\cos(\theta)) + D\theta\sin(\theta),
\end{equation}
with
\begin{equation}
  C = \frac{\beta\sin(\beta)}{\beta^{2} - \sin^{2}(\beta)},
  \quad
  D = \frac{\beta\cos(\beta) - \sin(\beta)}{\beta^{2} - \sin^{2}(\beta)}.
\end{equation}
Here $\beta = \pi/4$ is the angle between $\Gamma_{1}$
and~$\Gamma_{2}$. In Figure~\ref{fig:2Dsimulation} we show the
computed streamlines of the magma and matrix velocity fields for this
problem.
\begin{figure}
\centering
\subfloat[S][$\alpha=1$]{
\includegraphics[width=0.4\textwidth]{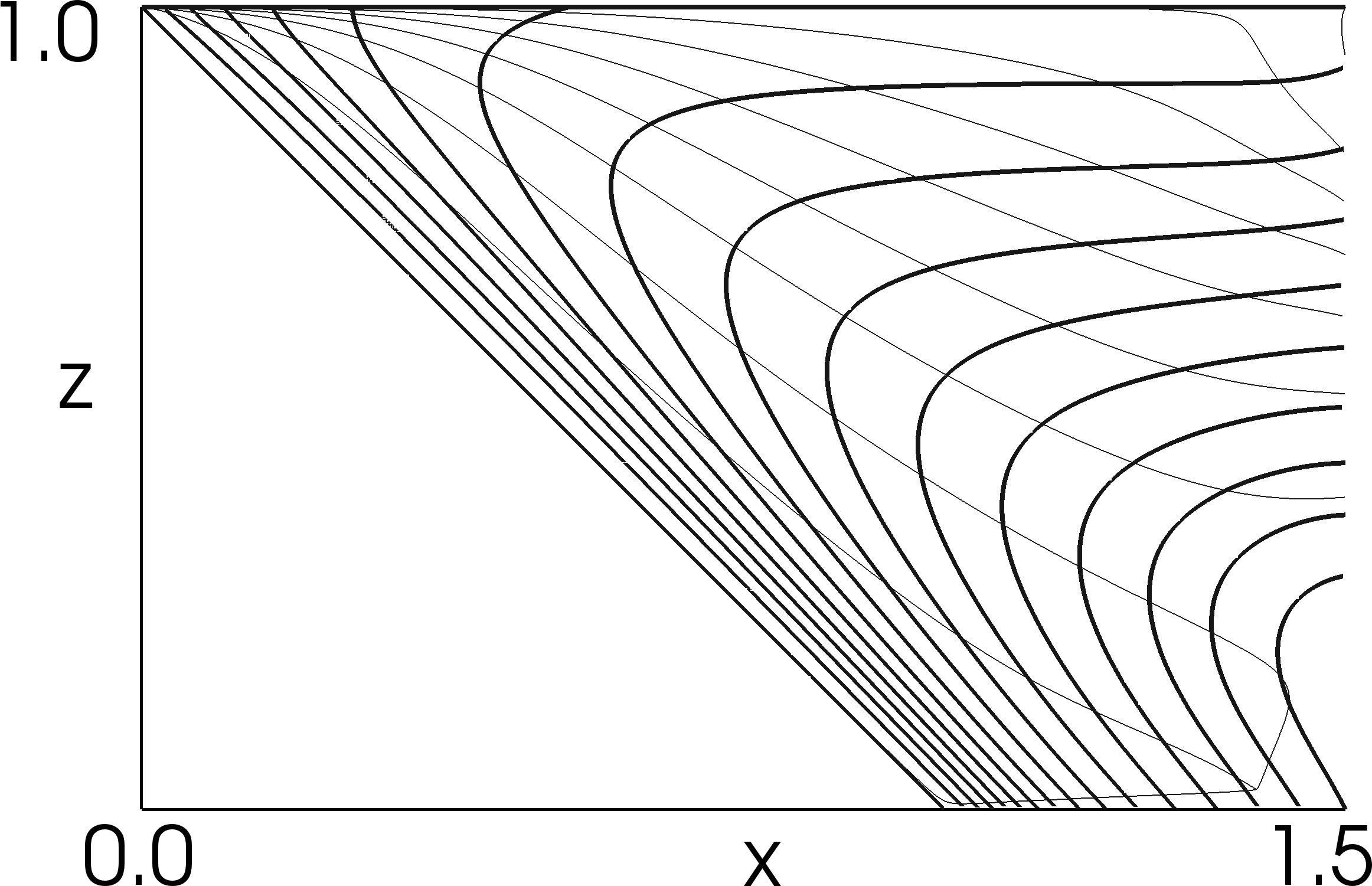}
\label{fig:2Dsimulationa}}
\
\subfloat[S][$\alpha=1000$]{
\includegraphics[width=0.4\textwidth]{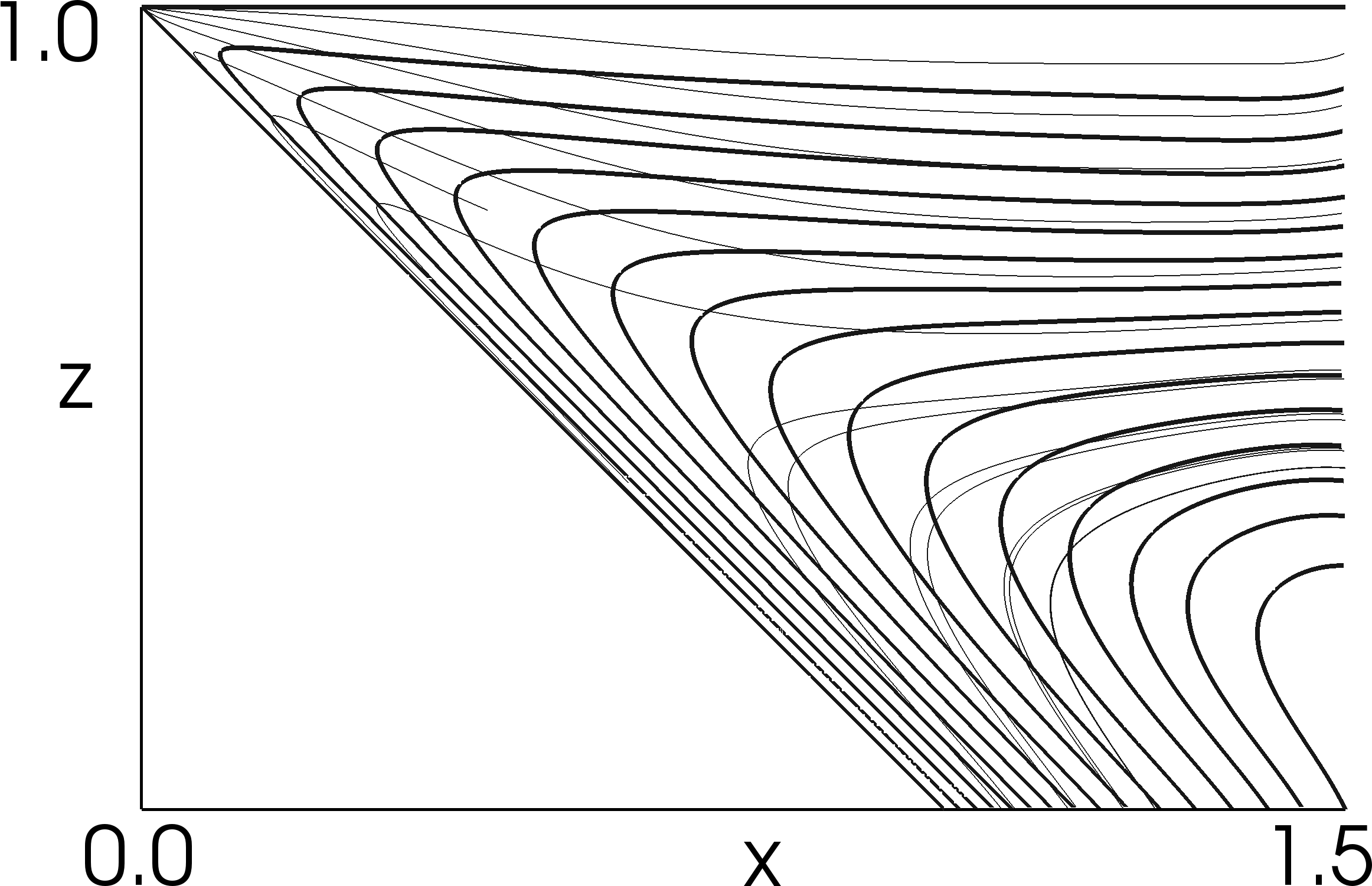}
\label{fig:2Dsimulationb}}
\caption{Streamlines of the magma (light) and matrix (dark) velocity
  fields in the wedge of a two-dimensional subduction zone using the
  corner flow boundary condition on $\Gamma_{3}$. The solution was
  computed on a mesh with 116176 elements.}
\label{fig:2Dsimulation}
\end{figure}

Table~\ref{tab:corner-2d} presents the number of solver iterations for
the LU and AMG preconditioners for different values of~$\alpha$.  We
observe very similar behaviour as we saw for the test in
Section~\ref{ss:tc1}.  The LU preconditioner is optimal and
uniform. The AMG preconditioner again shows slight dependence on the
problem size, and as $\alpha$ is increased the iteration count grows.

\begin{table}
  \caption{Number of iterations required for the corner flow problem
    using LU and AMG preconditioned MINRES for different levels of
    mesh refinement and varying~$\alpha$. For the $\alpha = 1000$
    case, four applications of a Chebyshev smoother, with one
    symmetric Gauss-Seidel iteration for each application, was used.}
\begin{center}
\begin{tabular}{c|cc|cc|cc|cc}
  & \multicolumn{2}{c|}{$\alpha=1$} & \multicolumn{2}{c|}{$\alpha=10$}
  & \multicolumn{2}{c|}{$\alpha=100$} & \multicolumn{2}{c}{$\alpha=1000$} \\
$N$        & LU & AMG   & LU & AMG   & LU & AMG   & LU & AMG$^{*}$ \\
\hline
 34,138     & 26 & 69    & 30 & 140    & 30 & 367   & 28 & 572 \\
 133,777    & 26 & 75    & 29 & 151    & 27 & 390   & 27 & 669 \\
 526,719    & 24 & 81    & 29 & 171    & 26 & 446   & 27 & 758 \\
\end{tabular}
\label{tab:corner-2d}
\end{center}
\end{table}

\subsubsection{Traction-free problem}

For the traction-free problem, instead of prescribing $\mathbf{u}_{\rm
  corner}$, we prescribe the zero-traction boundary condition,
$(\boldsymbol{\epsilon}(\mathbf{u}) -p \mathbf{I} + \alpha \nabla
\cdot \mathbf{u} \mathbf{I}\big) \cdot \mathbf{n} = 0$ on $\Gamma_3$.
Figure~\ref{fig:2Dsimulation_nostress} shows the computed streamlines
of the magma and matrix velocity fields for this problem.
\begin{figure}
\centering
\subfloat[S][$\alpha=1$]{
\includegraphics[width=0.4\textwidth]{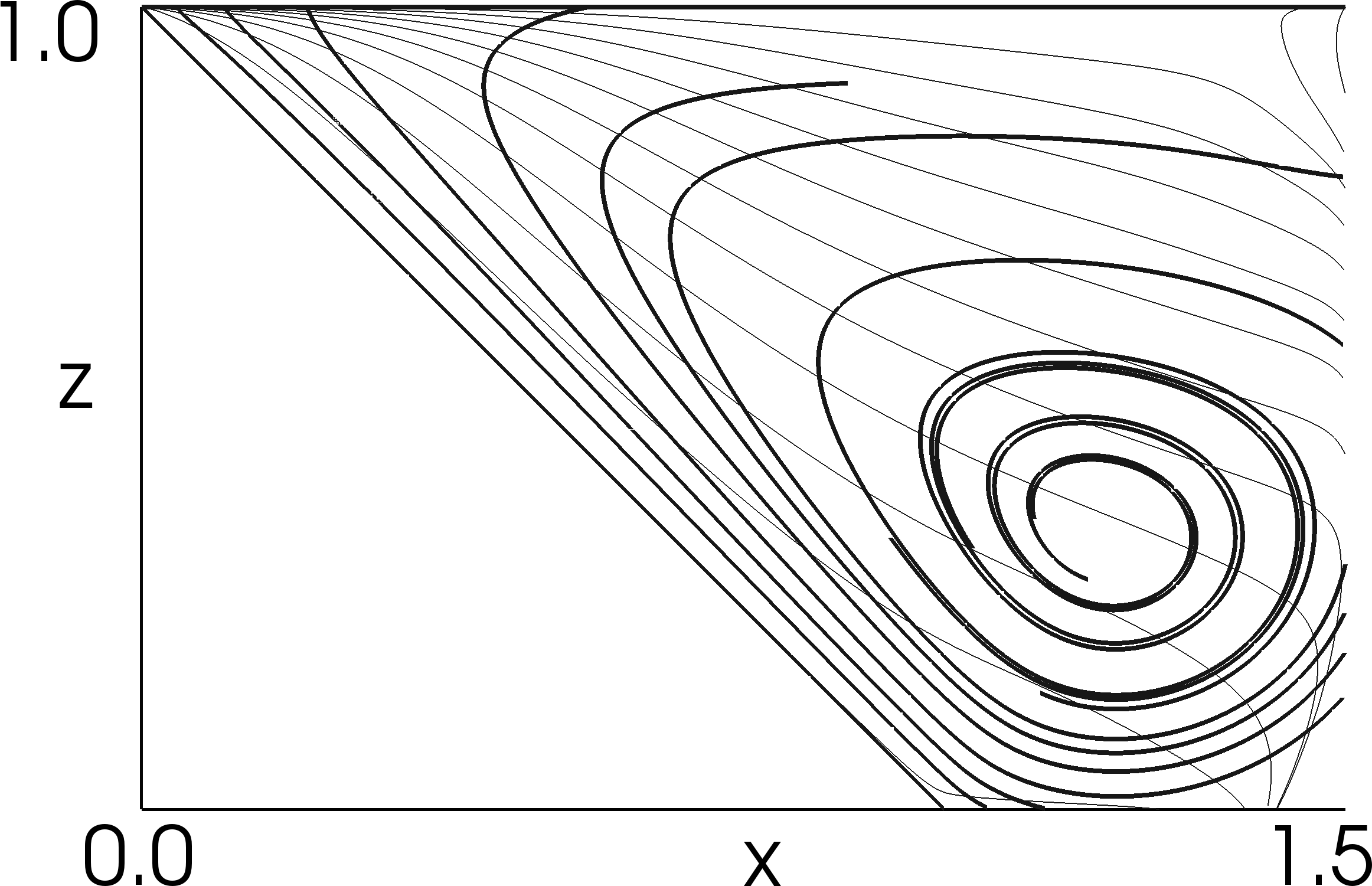}
\label{fig:2Dsimulationa_nostress}}
\
\subfloat[S][$\alpha=1000$]{
\includegraphics[width=0.4\textwidth]{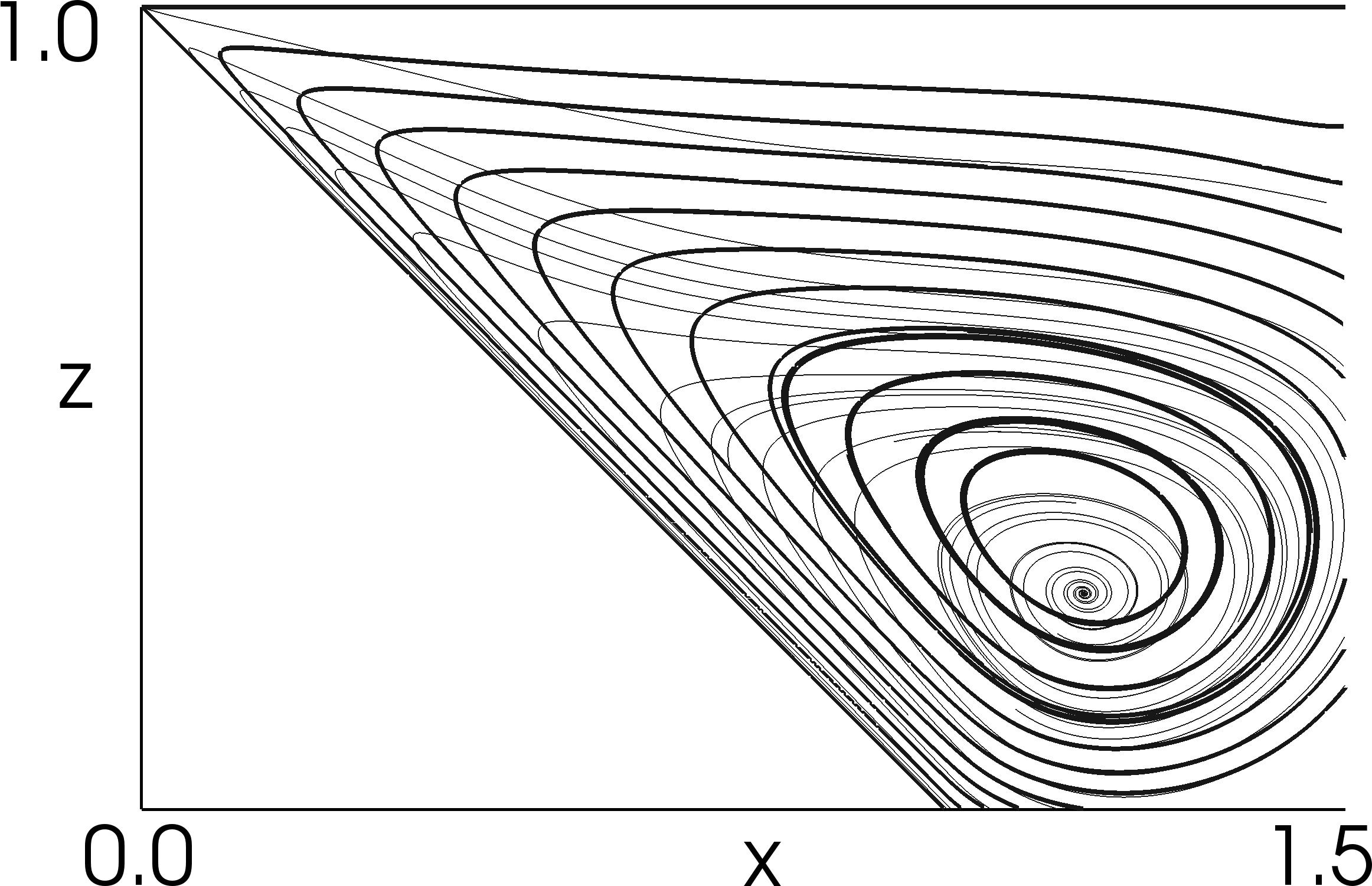}
\label{fig:2Dsimulationb_nostress}}
\caption{Streamlines of the magma (light) and matrix (dark) velocity
  fields in the wedge of a 2D subduction zone using no stress boundary
  conditions on~$\Gamma_{3}$. The solution was computed on a mesh with
  116176 elements.}
\label{fig:2Dsimulation_nostress}
\end{figure}

The solver iteration counts for this problem with different levels of
mesh refinement and for different values of $\alpha$ are presented in
Table~\ref{tab:no-stress-2d}. As for the analytic corner flow problem
of Section~\ref{ss:corner}, the LU-based preconditioner is optimal and
uniform. As expected, using the AMG-based preconditioner, the solver
is not uniform with respect to~$\alpha$.

\begin{table}
  \caption{Number of iterations to reach a relative tolerance of
    $10^{-8}$ using LU and AMG preconditioned MINRES for different
    values of $\alpha$ for the no-stress test.  For the $\alpha =
    1000$ case, four applications of a Chebyshev smoother, with one
    symmetric Gauss-Seidel iteration for each application, was used.}
\begin{center}
\begin{tabular}{c|cc|cc|cc|cc}
  & \multicolumn{2}{c|}{$\alpha=1$} & \multicolumn{2}{c|}{$\alpha=10$} & \multicolumn{2}{c|}{$\alpha=100$} & \multicolumn{2}{c}{$\alpha=1000$} \\
$N$         & LU & AMG   & LU & AMG   & LU & AMG   & LU & AMG$^{*}$ \\
\hline
 34,138     & 24 & 65    & 29 & 143   & 27 & 375   & 25 & 626 \\
 133,777    & 23 & 73    & 27 & 159   & 27 & 424   & 24 & 718 \\
 526,719    & 23 & 80    & 26 & 175   & 27 & 475   & 24 & 798
\end{tabular}
\label{tab:no-stress-2d}
\end{center}
\end{table}

\subsection{Magma dynamics problem in three dimensions}
\label{ss:tc3}

In the final case we test the solver for a three-dimensional problem
that is geometrically representative of a subduction zone.  We solve
\eqref{eq:magma} and \eqref{eq:bc} on the domain $\Omega$ depicted in
Figure~\ref{fig:3Dsubduction}. We set $L_{x}^{t} = 1.5$, $L_{x}^{b} =
0.5$, $L_{y} = 1$ and~$L_{z} = 1$, and use unstructured meshes of
tetrahedral cells.  Again we set the permeability as $k = 0.9(1 +
\tanh(-2r))$, with $r = \sqrt{x^{2} + z^{2}}$, and the porosity~$\phi
= 0.01$.
\begin{figure}
\centering
\includegraphics[width=0.5\textwidth]{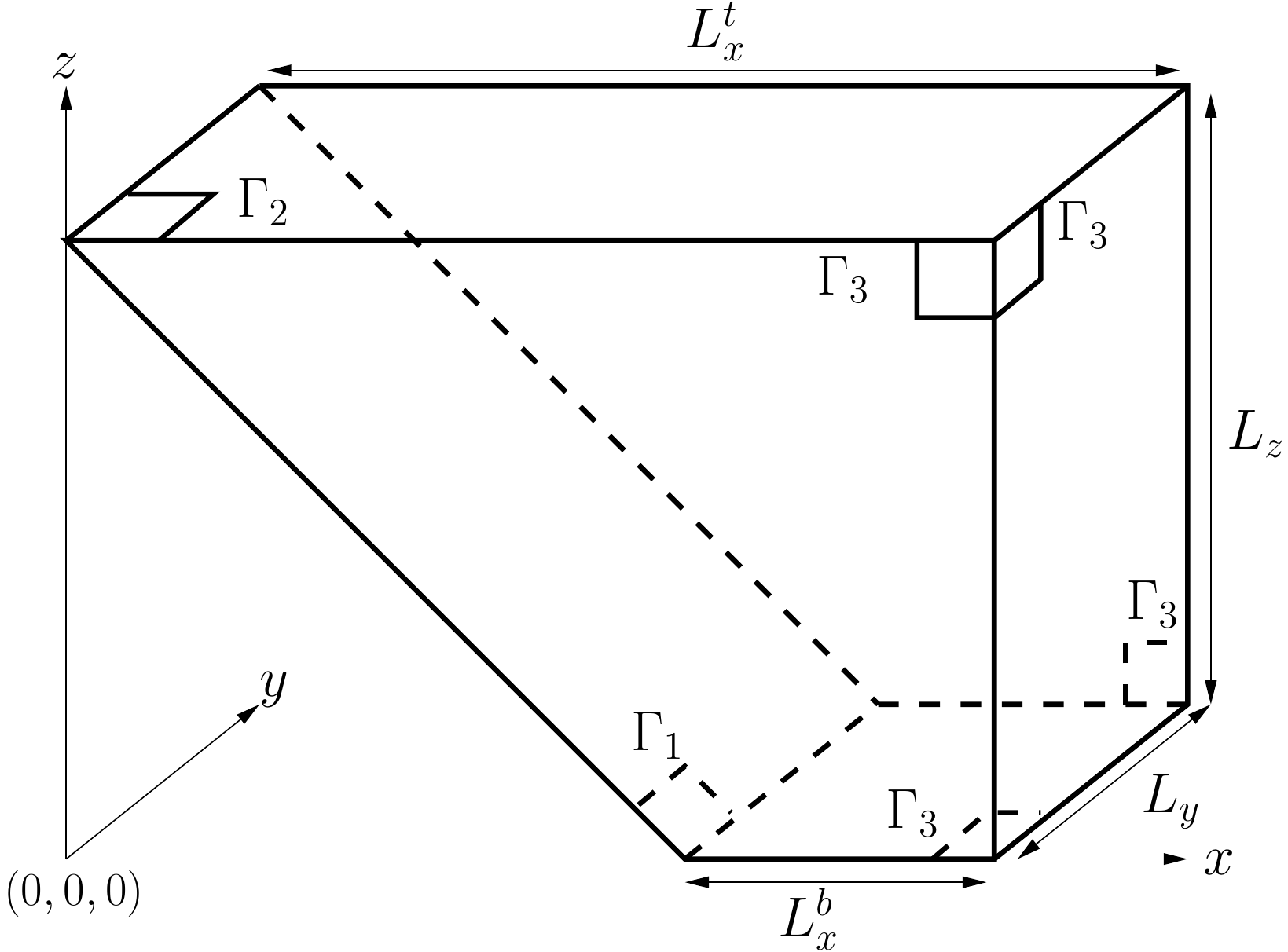}
\caption{Description of the wedge in a three-dimensional subduction
  zone.}
\label{fig:3Dsubduction}
\end{figure}

As boundary conditions, we prescribe $\mathbf{u} = \mathbf{u}_{\rm
  slab} = (1, 0.1, -1)/\sqrt{2}$ on $\Gamma_1$, $\mathbf{u} =
\mathbf{0}$ on $\Gamma_{2}$, $\big(\boldsymbol{\epsilon}(\mathbf{u}) -
p \mathbf{I} + \alpha \nabla \cdot \mathbf{u} \mathbf{I} \big) \cdot
\mathbf{n} = 0$ on $\Gamma_3$ and $-k \del{\nabla p - \mathbf{e}_{3}}
\cdot \mathbf{n} = 0$ on $\partial \Omega$. In
Figure~\ref{fig:3Dsimulation} we show computed vector plots of the
matrix and magma velocities for $\alpha = 1$ and~$\alpha = 1000$.
\begin{figure}
\centering
\subfloat[S][Matrix velocity, $\alpha = 1$.]{
\includegraphics[width=0.4\textwidth]{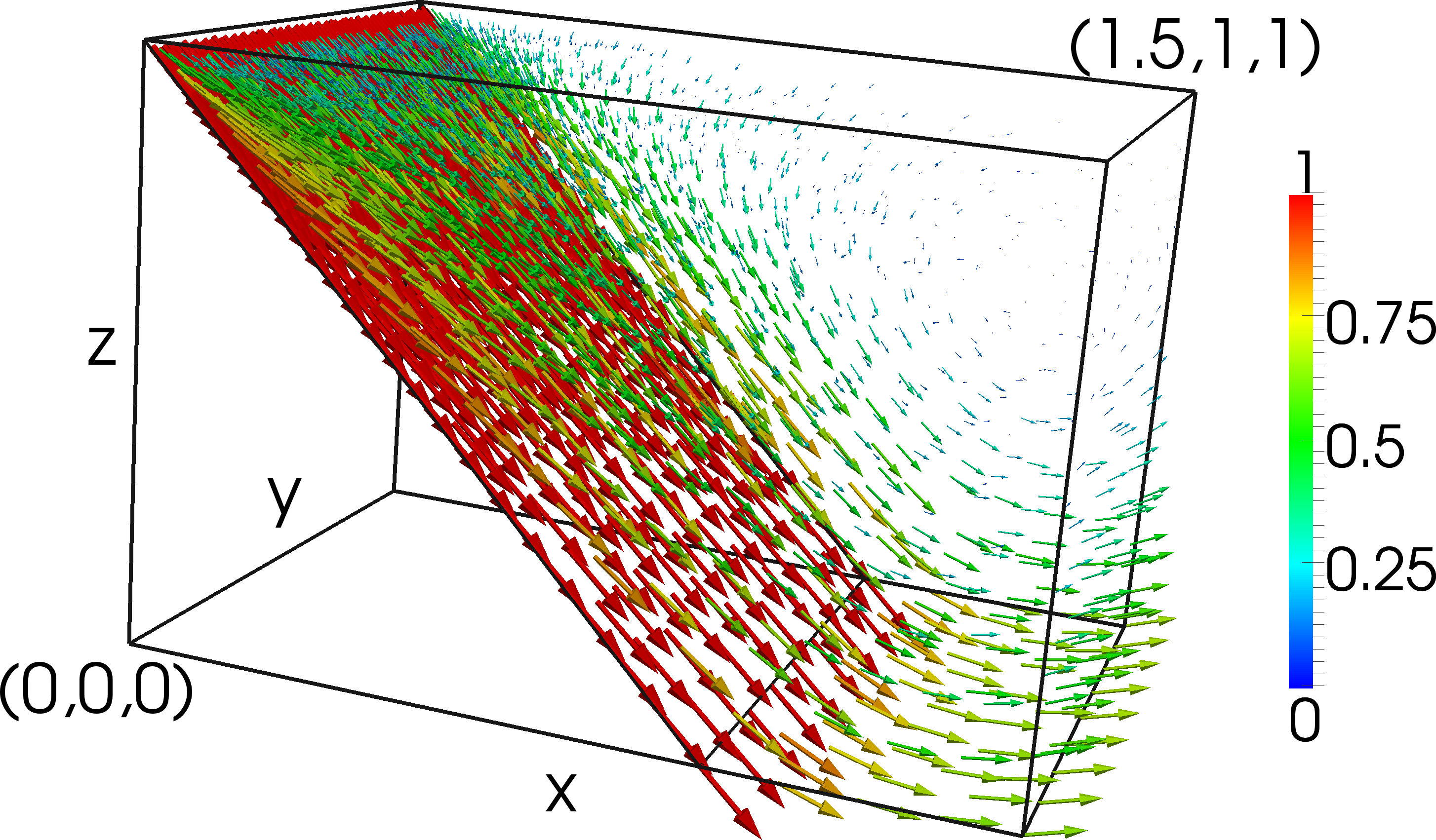}
\label{fig:3Dsimulationa}}
\
\subfloat[S][Magma velocity, $\alpha = 1$.]{
\includegraphics[width=0.4\textwidth]{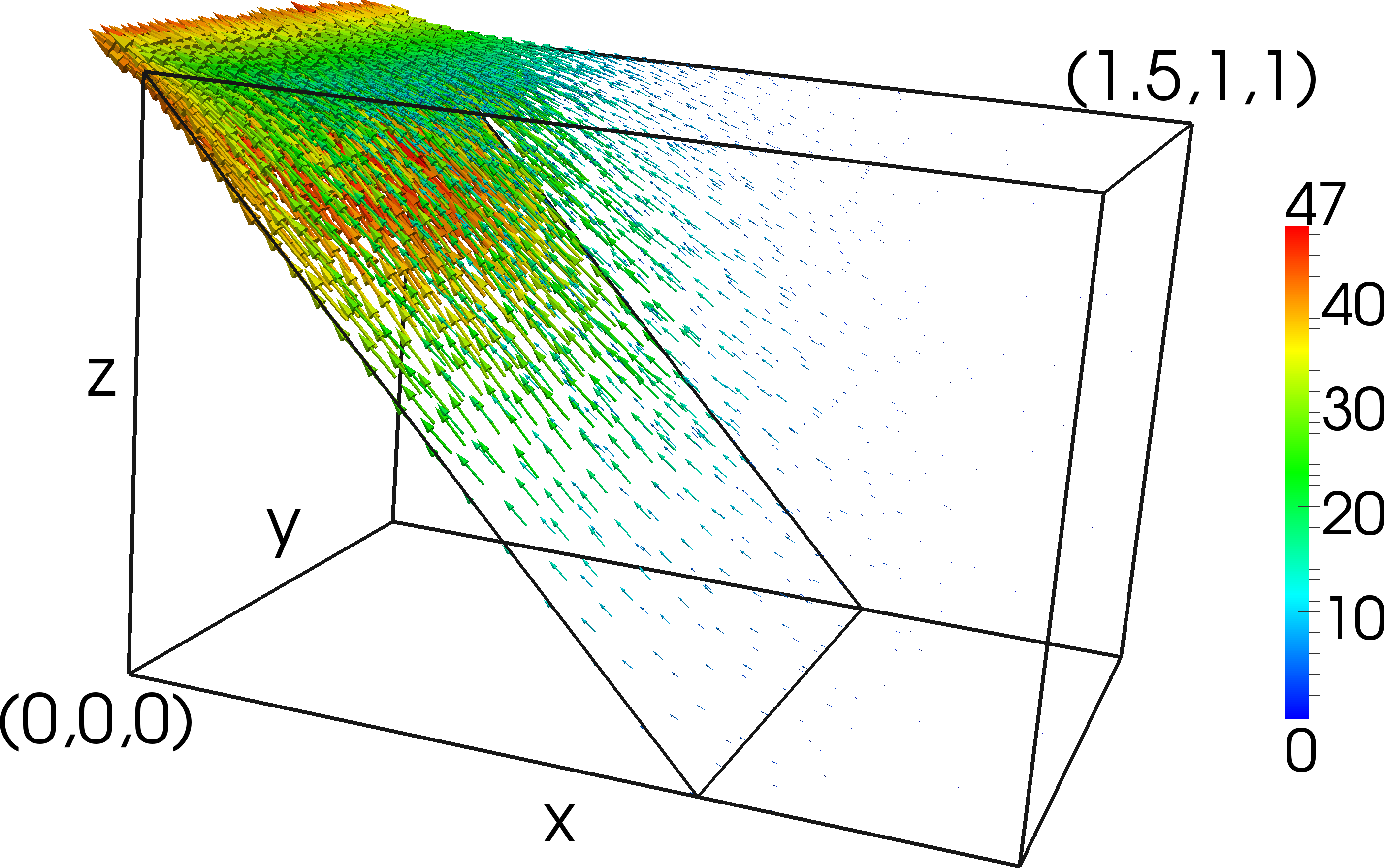}
\label{fig:3Dsimulationb}}
\\
\subfloat[S][Matrix velocity, $\alpha = 1000$.]{
\includegraphics[width=0.4\textwidth]{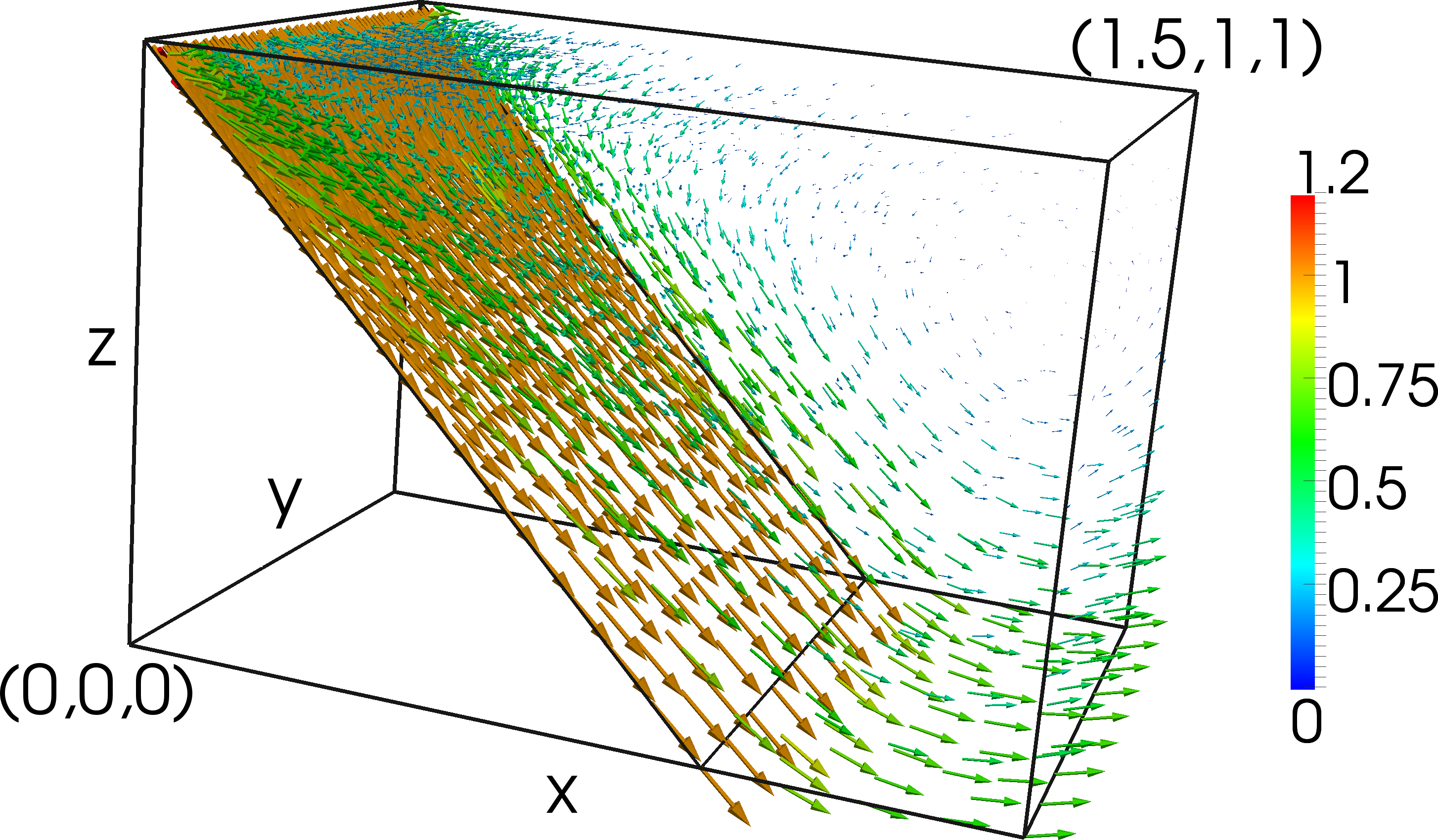}
\label{fig:3Dsimulationc}}
\
\subfloat[S][Magma velocity, $\alpha = 1000$.]{
\includegraphics[width=0.4\textwidth]{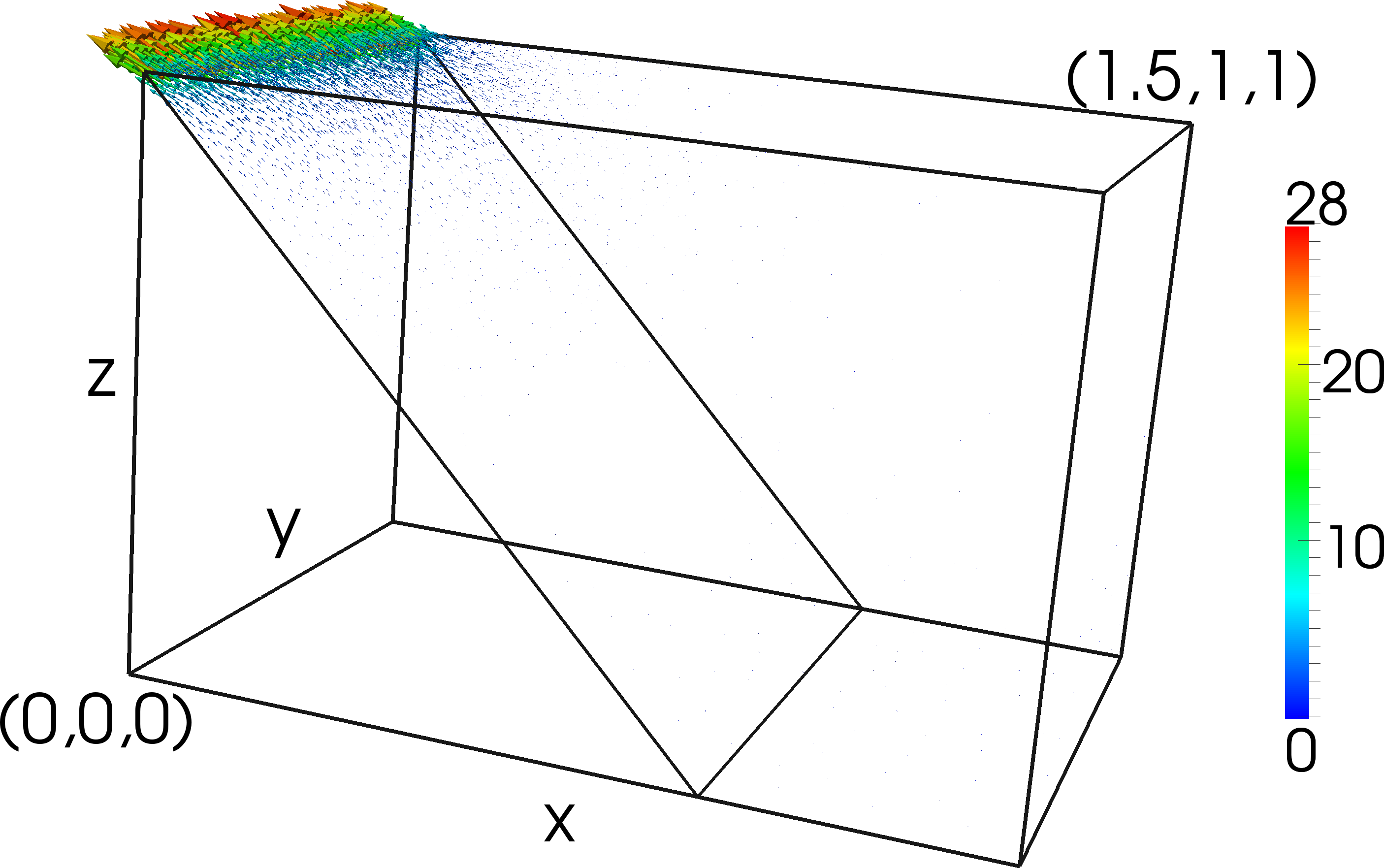}
\label{fig:3Dsimulationd}}
\caption{Vector plots of the magma and matrix velocities in the wedge
  of a three-dimensional subduction zone for $\alpha = 1$ and $\alpha
  = 1000$ using the stress-free boundary conditions on~$\Gamma_{3}$.}
\label{fig:3Dsimulation}
\end{figure}

Table~\ref{tab:tc3} shows the number of iterations needed for the AMG
preconditioned MINRES method for the three-dimensional wedge problem.
The LU preconditioned solver is not practical for this problem when
using reasonable mesh resolutions.  All cases have been computed in
parallel using 16 processes.  The computed examples span a range of
problem sizes, and only relatively small changes in the iteration
count are observed for changes in the number of degrees of
freedom. Again, as $\alpha$ becomes larger, so too does the iteration
count.
\begin{table}
  \caption{Number of iterations required for AMG preconditioned MINRES
    for the three-dimensional subduction model for different levels of
    mesh refinement and and different values of~$\alpha$. The number
    of degrees of freedom is denoted by~$N$. For the $\alpha = 1000$
    case, four applications of a Chebyshev smoother, with one
    symmetric Gauss-Seidel iteration for each application, was used.
    All tests were run using 16 MPI processes.}
\begin{center}
\begin{tabular}{c|c|c|c|c}
$N$        & $\alpha=1$ & $\alpha = 10$ & $\alpha = 100$  & $\alpha=1000$  \\
\hline
 88,500    & 42         & 127           & 363             & 654   \\
 400,690   & 44         & 122           & 355             & 692   \\
 1,821,991 & 43         & 122           & 367             & 732   \\
 8,124,691 & 41         & 120           & 355             & 775
\end{tabular}
\label{tab:tc3}
\end{center}
\end{table}

\section{Conclusions}
\label{s:conclusions}

In this work we introduced and analysed an optimal preconditioner for
a finite element discretisation of the simplified McKenzie equations
for magma/mantle dynamics. Analysis of the preconditioner showed that
the Schur complement of the block matrix arising from the finite
element discretisation of the simplified McKenzie equations may be
approximated by a pressure mass matrix plus a permeability matrix. The
analysis was verified through numerical simulations on a unit square
and two- and three-dimensional wedge flow problems inspired by
subduction zones. For all computations we used $P^{2}$--$P^{1}$
Taylor--Hood finite elements as they are inf-sup stable in the
degenerate limit of vanishing permeability.  Numerical tests
demonstrated optimality of the solver.  We observed that the multigrid
version of the preconditioner was not uniform with respect to the
bulk-to-shear-viscosity ratio~$\alpha$. As $\alpha$ is increased, the
iteration count for the solver increases. We observe a similar
behaviour as $k^{*}$ increases.

The analysis and testing of an optimal block preconditioning method
for magma/mantle dynamics presented in this work lays a basis for
creating efficient and optimal simulation tools that will ultimately
be put to use to study the genesis and transport of magma in
plate-tectonic subduction zones. Optimality has been demonstrated, but
some open questions remain regarding uniformity with respect to some
model parameters.

\section*{Acknowledgements}

We thank L.~Alisic and J.~F.~Rudge for the many discussions held
related to this paper. We also thank the reviewers M.~Knepley,
M.~Spiegelman, C.~Wilson and one that remained anonymous, whose
comments helped improve this paper. The authors acknowledge the
support of the Natural Environment Research Council under grants
NE/I026995/1 and NE/I023929/1. Katz is furthermore grateful for the
support of the Leverhulme Trust.

\bibliographystyle{abbrvnat}
\bibliography{references}
\end{document}